\documentclass{article}

\usepackage[preprint,nonatbib]{neurips_2019}

\usepackage[utf8]{inputenc} 
\usepackage[T1]{fontenc}    
\usepackage{hyperref}       
\usepackage{url}            
\usepackage{booktabs}       
\usepackage{amsfonts}       
\usepackage{nicefrac}       
\usepackage{microtype}      

\usepackage{amssymb,graphics,amsmath,amsthm,amsopn,amstext,amsfonts,color,fullpage,fancybox,hyperref,bm}
\usepackage{subfig,placeins} 
\usepackage{graphicx,color,float,epstopdf}
\usepackage{fix-cm}
\usepackage{arydshln}
\usepackage{mathtools}
\usepackage{graphicx}
\usepackage{psfrag}
\usepackage{epsfig}
\usepackage{svg}
\setsvg{inkscape={"C:/Program Files/Inkscape/inkscape.com"}}
\usepackage{graphics,bm,paralist,ifthen,color, algorithm}
\usepackage{array}
\usepackage{algpseudocode}
\usepackage{calc}
\usepackage{longtable}
\usepackage{mathrsfs}
\usepackage{float}
\usepackage{lipsum}
\usepackage{epstopdf}
\usepackage{arydshln, colortbl, diagbox, lettrine, marginnote, picinpar, sidecap, booktabs, multirow, ulem}
\usepackage{caption} 
\makeatletter
\def\BState{\State\hskip-\ALG@thistlm}
\makeatother

\algnewcommand{\IfThenElse}[3]{
	\State \algorithmicif\ #1\ \algorithmicthen\ #2\ \algorithmicelse\#3}

\DeclareMathOperator*{\argmax}{arg\,max}
\DeclareMathOperator*{\argmin}{arg\,min}

\newcommand{\bx}{{\mathbf{x}}}
\newcommand{\by}{{\mathbf{y}}}

\newcommand{\bO}{{\mathcal{O}}}
\newcommand{\cSt}{\Theta}
\newcommand{\cSa}{\mathcal{A}}
\newcommand{\cX}{{\cal X}}
\newcommand{\cY}{{\cal Y}}
\newcommand{\bbalpha}{\bar{\boldsymbol{\alpha}}}

\newcommand{\proj}{\rm proj}
\newcommand{\st}{\rm s.t.}
\newcommand{\fl}{f_{\lambda}}
\newcommand{\gl}{g_{\lambda}}
\newcommand{\btheta}{\boldsymbol{\theta}}
\newcommand{\balpha}{\boldsymbol{\alpha}}

\newtheorem{example}{Example}[section]
\newtheorem{theorem}{Theorem}[section]

\newtheorem{lemma}[theorem]{Lemma}
\newtheorem{corollary}[theorem]{Corollary}
\newtheorem{definition}[theorem]{Definition}
\newtheorem{remark}[theorem]{Remark}
\newtheorem{assumption}[theorem]{Assumption}

\title{Solving a Class of Non-Convex Min-Max Games Using Iterative First Order Methods}

\author{Maher Nouiehed \\{nouiehed@usc.edu} \thanks{Department of Industrial and Systems Engineering, University of Southern California}
\And
Maziar Sanjabi\\{sanjabi@usc.edu} \thanks{Data Science and Operations Department, Marshall School of Business, University of Southern California}
\And
Tianjian Huang \\{tianjian@usc.edu} \thanks{Department of Industrial and Systems Engineering, University of Southern California}
\And
Jason D. Lee\\{jasonlee@princeton.edu} \thanks{Department of Electrical Engineering, Princeton University}
\And
Meisam Razaviyayn\\{razaviya@usc.edu} \thanks{Department of Industrial and Systems Engineering, University of Southern California}}

\begin{document}

\maketitle

\begin{abstract} Recent applications that arise in machine learning have surged significant interest in solving min-max saddle point games. This problem has been extensively studied in the convex-concave regime for which a global equilibrium solution can be computed efficiently. In this paper, we study the problem in the non-convex regime and show that an $\varepsilon$--first order stationary point of the game can be computed when one of the player’s objective can be optimized to global optimality efficiently. In particular, we first consider the case where the objective of one of the players satisfies the Polyak-{\L}ojasiewicz (PL) condition. For such a game, we show that a simple multi-step gradient descent-ascent algorithm finds an $\varepsilon$--first order stationary point of the problem in $\widetilde{\mathcal{O}}(\varepsilon^{-2})$ iterations. Then we show that our framework can also be applied to the case where the objective of the ``max-player" is concave. In this case, we propose a multi-step gradient descent-ascent algorithm that finds an $\varepsilon$--first order stationary point of the game in $\widetilde{\cal O}(\varepsilon^{-3.5})$ iterations, which is the best known rate in the literature. We applied our algorithm to a fair classification problem of Fashion-MNIST dataset and observed that the proposed algorithm results in smoother training and better generalization. 
\end{abstract}

\section{Introduction}
Recent years have witnessed a wide range of machine learning and robust optimization applications being formulated as a min-max saddle point game; see \cite{sanjabi2018convergence, dai2018sbeed, dai2018kernel, rafique2018non, ghosh2018efficient, sinha2018certifying} and the references therein. 
Examples of problems that are formulated under this framework include generative adversarial networks (GANs) \cite{sanjabi2018convergence}, reinforcement learning \cite{dai2018sbeed}, adversarial learning \cite{sinha2018certifying}, learning exponential families \cite{dai2018kernel}, fair statistical inference~\cite{edwards2015censoring, xu2018fairgan, sattigeri2018fairness, madras2018learning}, generative adversarial imitation learning \cite{cai2019global,ho2016generative}, distributed non-convex optimization \cite{lu2019block} and many others. These applications require solving an optimization problem of the form
\begin{align}
\min_{\btheta \in \cSt}\; \max_{\balpha \in \cSa} \;\;f(\btheta,\balpha).
\end{align}
This optimization problem can be viewed as a zero-sum game between two players. The goal of the first player is to minimize  $f(\btheta, \balpha)$ by tuning $\btheta$, while the other player's objective is to maximize $f(\btheta, \balpha)$ by tuning $\balpha$. Gradient-based methods, especially gradient descent-ascent (GDA), are widely used in practice to solve these problems. GDA alternates between a gradient ascent steps on $\balpha$ and a gradient descent steps on $\btheta$. Despite its popularity, this algorithm fails to converge even for simple bilinear zero-sum games~\cite{mescheder2018training, mai2018cycles, daskalakis2018last, balduzzi2018mechanics, letcher2019differentiable}. This failure  was fixed by adding negative momentum or by using primal-dual methods proposed by \cite{gidel2018negative,gidel2018variational, chambolle2016ergodic, daskalakis2017training, daskalakis2018limit, liang2018interaction}. 




When the objective~$f$ is convex in $\btheta$ and concave in $\balpha$, the corresponding variational inequality becomes monotone.  This setting has been extensively studied and different algorithms have been developed for finding a Nash equilibrium 
\cite{nesterov2007dual, gidel2018variational, monteiro2010complexity, juditsky2016solving, mertikopoulos2018mirror, gidel2016frank, hamedani2018iteration, mokhtari2019unified, facchinei2007finite, nemirovski2004prox}. Moreover, \cite{dang2015convergence} proposed an algorithm for solving a more general setting that covers both monotone and psuedo-monotone variational problems.



While the convex-concave setting has been extensively studied in the literature, recent machine learning applications urge the necessity of moving beyond these classical settings. For example, in a typical GAN problem formulation, two neural networks (generator and discriminator) compete in a non-convex zero-sum game framework \cite{goodfellow2016deep}.
For general non-convex non-concave games, \cite[Proposition~10]{jin2019minmax} provides an example for which local Nash equilibrium does not exist. Similarly, one can show that even second-order Nash equilibrium may not exist for non-convex games, see Section~\ref{section:stationarity} for more details. Therefore, a well-justified objective is to find first order Nash equilibria of such games \cite{pang2016unified}; see definitions and discussion in Section~\ref{section:stationarity}.  The first order Nash equilibrium can be viewed as a direct extension of the concept of first order stationarity in optimization to the above min-max game. While $\varepsilon$--first order stationarity in the context of optimization can be found efficiently in ${\cal O}(\varepsilon^{-2})$ iterations with gradient descent algorithm \cite{nesterov2013introductory}, the question of whether it is possible to design a gradient-based algorithm that can find an $\varepsilon$--first order Nash equilibrium  for general non-convex saddle point games remains open.\\

%
%

Several recent results provided a partial answer to the problem of finding first-order stationary points of a non-convex min-max game. For instance, \cite{sanjabi2018convergence} proposed a stochastic gradient descent algorithm for the case when $f(\cdot,\cdot)$ is strongly concave in~$\balpha$ and show convergence of the algorithm to an $\varepsilon$--first-order Nash equilibrium  with $\widetilde{\cal O}(\varepsilon^{-2})$ gradient evaluations. Also, the work~\cite{jin2019minmax} analyzes the gradient descent algorithm with Max-oracle and shows $O(\epsilon^{-4})$ gradient evaluations and max-oracle calls for solving min-max problems where the inner problem can be solved in one iteration using an existing oracle.
More recently, \cite{lu2019block, lu2019hybrid} considered the case where $f$ is concave in $\balpha$. They developed a  descent-ascent algorithm  with iteration complexity $\widetilde{\cal O}(\varepsilon^{-4})$. 
In this non-convex concave setting, \cite{rafique2018non} proposed a  stochastic sub-gradient descent method with worst-case complexity $\widetilde{\cal O}(\varepsilon^{-6})$. Under the same concavity assumption on $f$, in this paper, we propose an alternative multi-step framework that finds an $\varepsilon$--first order Nash equilibrium/stationary with $\widetilde{\cal O}(\varepsilon^{-3.5})$ gradient evaluations. 

In an effort to solve the more general non-convex non-concave setting, \cite{lin2018solving} developed a framework that converges to $\varepsilon$-first order stationarity/Nash equilibrium under the assumption that there exists a solution to the Minty variational inequality at each iteration. Although among the first algorithms with have theoretical convergence guarantees in the non-convex non-concave setting, the conditions required are strong and difficult to check. To the best of our knowledge, there is no practical problem for which the Minty variational inequality condition has been proven. 
With the motivation of exploring the non-convex non-concave setting, we propose a simple multi-step gradient descent ascent algorithm for the case where the objective of one of the players satisfies the Polyak-{\L}ojasiewicz (PL) condition. We show  the worst-case complexity of  $\widetilde{\cal O}(\varepsilon^{-2})$ for our algorithm. This rate is  optimal in terms of dependence on $\varepsilon$ up to logarithmic factors as discussed in Section~\ref{section:PL-Games}. 
Compared to Minty variational inequality condition used in \cite{lin2018solving}, the PL condition is very well studied in the literature and has been theoretically verified for objectives of optimization problems arising in many practical problems. For example, it has been proven to be true for objectives of over-parameterized deep networks \cite{du2018gradientb},  learning LQR models \cite{fazel2018global}, phase retrieval \cite{sun2018geometric}, and many other simple problems discussed in \cite{PL_karimi_2016}. In the context of min-max games, it has also been proven useful in generative adversarial imitation learning with LQR dynamics \cite{cai2019global}, as discussed in Section~\ref{section:PL-Games}.

The rest of this paper is organized as follows. In Section~\ref{section:stationarity} we define the concepts of  First-order Nash equilibrium (FNE) and $\varepsilon$--FNE. In Section~\ref{section:PL-Games}, we describe our algorithm designed for min-max games with the objective of one player satisfying the PL condition. Finally, in Section~\ref{section:Non-Convex-Concave} we describe our method for solving games in which the function $f(\theta, \alpha)$ is concave in $\alpha$ (or convex in $\theta$).

\section{Two-player Min-Max Games and First-Order Nash Equilibrium}\label{section:stationarity}
Consider the two-player  zero sum min-max game
\begin{align}
\min_{\btheta \in \cSt}\; \max_{\balpha \in \cSa} \;\;f(\btheta,\balpha),\label{eq: game1-cons}
\end{align}
where $\cSt$ and $\cSa$ are both convex sets, and $f(\btheta, \balpha)$ is a  continuously differentiable function. We say $(\btheta^*,\balpha^*)\in \cSt\times \cSa$ is a \textit{Nash equilibrium} of the game if 
\[ f(\btheta^*, \balpha) \leq f(\btheta^*, \balpha^*) \leq f(\btheta, \balpha^*) \quad \forall \, \btheta \in \cSt, \,\, \forall \, \balpha \in \cSa.\]
In convex-concave games, such a Nash equilibrium always exists \cite{jin2019minmax} and several algorithms were proposed to find Nash equilibria \cite{gidel2016frank, hamedani2018iteration}. However, in the non-convex non-concave regime, computing these points is in general NP-Hard. 
In fact, even finding  local Nash equilibria is NP-hard in the general non-convex non-concave regime.In addition, as shown by~\cite[Proposition~10]{jin2019minmax}, local Nash equilibria for general non-convex non-concave games may not exist.  Thus, in this paper we aim for the  less ambitious goal of finding \textit{first-order Nash equilibrium} which is defined in the sequel.
\begin{definition}[FNE] 
\label{def.FNE}
A point $(\btheta^*,\balpha^*)\in \cSt\times \cSa$ is a \textit{first order Nash equilibrium} (FNE) of the game~\eqref{eq: game1-cons} if  
\begin{equation}
\label{eq:FNEcond}
\langle \nabla_{\btheta} f(\btheta^*, \balpha^*), \btheta -\btheta^* \rangle \geq 0 \quad \forall \; \btheta \in \cSt \quad   {\rm   and  }
\quad \langle \nabla_{\balpha} f(\btheta^*, \balpha^*), \balpha -\balpha^* \rangle \leq 0 \quad \forall \;\balpha \in \cSa.
\end{equation}
\end{definition}
Notice that this definition, which is also used in \cite{pang2016unified, pang2011nonconvex}, contains the first order necessary optimality conditions of the objective function of each player \cite{bertsekas1999nonlinear}.  Thus they are necessary conditions for local Nash equilibrium.  Moreover, in the absence of constraints, the above definition simplifies to $\nabla_{\btheta} f(\btheta^*, \balpha^*)=0$   and   $\nabla_{\balpha} f(\btheta^*, \balpha^*)=0$, which are the well-known unconstrained first-order optimality conditions. 
Based on this observation, it is tempting to think that the above first-order Nash equilibrium condition does not differentiate between the min-max type solutions of \eqref{eq: game1-cons} and min-min solutions of the type $\min_{\btheta \in \cSt, \balpha\in \cSa} f(\btheta,\balpha)$. However, the direction of the second inequality in \eqref{eq:FNEcond} would be different if we have considered the min-min problem instead of min-max problem. This different direction makes the problem of finding a FNE non-trivial. The following theorem guarantees the existence of first-order Nash equilibria under some mild assumptions.
\begin{theorem}[Restated from Proposition~2 in \cite{pang2016unified}]
    Suppose the sets $\cSt$ and $\cSa$ are no-empty, compact, and convex. Moreover, assume that the function $f(\cdot,\cdot)$ is twice continuously differentiable. Then there exists a feasible point $(\bar{\btheta},\bar{\balpha})$ that is first-order Nash equilibrium.
\end{theorem}



The above theorem guarantees existence of FNE points even when (local) Nash equilibria may not exist. The next natural question is about the computability of such methods. Since in practice we use iterative methods for computation, we need to define the notion of approximate--FNE.

\begin{definition}[Approximate FNE]\label{def:cons-approx-stationarity}
	A point $(\btheta^*, \balpha^*)$ is said to be an $\varepsilon$--first-order Nash equilibrium ($\varepsilon$--FNE) of the game~\eqref{eq: game1-cons} if
	\[
	\cX(\btheta^*, \balpha^*) \leq \varepsilon \quad \mbox{and} \quad \cY(\btheta^*, \balpha^*) \leq \varepsilon,
	\]
	where 
	\begin{equation}\label{eq:X_k2}
	\cX(\btheta^*,\balpha^*) \triangleq - \min_{\btheta}\,\, \langle \nabla_{\btheta} f(\btheta^*, \balpha^*), \btheta -\btheta^* \rangle  \;\;  {\st} \; \btheta  \in \cSt, \, \|\btheta -\btheta^*\|\leq 1,
	\end{equation}
	and
	\begin{equation}\label{eq:Y_k2}
	\cY(\btheta^*,\balpha^*) \triangleq \max_{\balpha}\,\, \langle \nabla_{\balpha} f(\btheta, \balpha), \balpha - \balpha^*\rangle  \;\; {\st} \;\;  \balpha \in \cSa, \, \|\balpha - \balpha^*\|\leq 1.
	\end{equation}
\end{definition}

In the absence of constraints, $\varepsilon$--FNE in Definition~\ref{def:cons-approx-stationarity} reduces to 
$\|\nabla_{\btheta} f(\btheta^*, \balpha^*)\|\leq \varepsilon \quad {\rm and } \quad  \|\nabla_{\balpha} f(\btheta^*, \balpha^*)\| \leq \varepsilon.$

\begin{remark}
    The $\varepsilon$--FNE definition above is based on the first order optimality measure of the objective of each player. Such first-order optimality measure has been used before in the context of optimization; see~\cite{conn2000trust}. Such a condition guarantees that each player cannot improve their objective function using first order information. Similar to the optimization setting, one can define the second-order Nash equilibrium as a point that each player cannot improve their objective further by using first and second order information of their objectives. However, the use of second order Nash equilibria is more subtle in the context of games. The following example shows that such a point may not exist. Consider the game
\[\min_{-1 \leq \theta \leq 1} ~\max_{-2 \leq \alpha \leq 2} - \theta^2 + \alpha^2 +4\theta \alpha.\]
Then $(0,0)$ is the only first-order Nash equilibrium and is not a second-order Nash equilibrium. 
\end{remark}

In this paper, our goal is to find an $\varepsilon$--FNE of the game~\eqref{eq: game1-cons} using iterative methods. To proceed, we make the following standard assumptions about the smoothness of the objective function~$f$.
\begin{assumption}
	\label{assumption: LipSmooth-uncons}
	 The function~$f$ is continuously differentiable in both $\btheta$ and $\balpha$ and there exists constants $L_{11}$, $L_{22}$ and $L_{12}$ such that for every $\balpha,\balpha_1,\balpha_{2}\in \cSa$, and $\btheta,\btheta_1,\btheta_2 \in \cSt$, we have
	 \[\small
	\begin{array}{ll}
	\|\nabla_{\btheta} f(\btheta_1,\balpha)-\nabla_{\btheta} f(\btheta_2,\balpha)\|\leq L_{11}\|\btheta_1-\btheta_2\|,
	& \|\nabla_{\balpha} f(\btheta,\balpha_1)-\nabla_{\balpha} f(\btheta,\balpha_{2})\|\leq L_{22}\|\balpha_1-\balpha_{2}\|,\\
	\|\nabla_{\balpha} f(\btheta_1,\balpha)-\nabla_{\balpha} f(\btheta_2,\balpha)\|\leq L_{12}\|\btheta_1-\btheta_2\|,
	& \|\nabla_{\btheta} f(\btheta,\balpha_1)-\nabla_{\btheta} f(\btheta,\balpha_{2})\|\leq L_{12}\|\balpha_1-\balpha_{2}\|.\\
	\end{array}  
	\normalsize
	\]
	\normalsize
\end{assumption}
\normalsize




\section{Non-Convex PL-Game} \label{section:PL-Games}
In this section, we consider the problem of developing an ``efficient" algorithm  for finding an $\varepsilon$--FNE of~\eqref{eq: game1-cons} when the objective of one of the players satistys Polyak-{\L}ojasiewicz (PL) condition. To proceed, let us first formally define the Polyak-{\L}ojasiewicz (PL) condition.

\begin{definition}[Polyak-{\L}ojasiewicz Condition]\label{def: PL}
	A differentiable function $h(\bx)$ with the minimum value $h^* = \min_x\; h(\bx)$ is said to be $\mu$-Polyak-{\L}ojasiewicz ($\mu$-PL) if
	\begin{align}
	\frac{1}{2}\|\nabla h(\bx)\|^2\geq \mu (h(\bx)-h^*),\quad \forall x.
	\end{align}
\end{definition}

The PL-condition has been established and utilized for analyzing many practical modern problems \cite{PL_karimi_2016, fazel2018global,du2018gradientb,sun2018geometric, cai2019global}. Moreover, it is well-known that a function can be non-convex and still satisfy the PL condition~\cite{PL_karimi_2016}. Based on the definition above, we define a class of min-max PL-games.

\begin{definition}[PL-Game]\label{def: PLGame}
	We say that the min-max game~\eqref{eq: game1-cons} is a PL-Game if the max player is unconstrained, i.e., $\cSa = \mathbb{R}^n$, and there exists a constant $\mu>0$ such that the function $h_{\btheta}(\balpha) \triangleq -f(\btheta, \balpha)$ is $\mu$-PL for any fixed value of~$\btheta\in \cSt$.
\end{definition}

A simple example of a practical PL-game is detailed next.

\begin{example}[Generative adversarial imitation learning of linear quadratic regulators]
Imitation learning is a paradigm that aims to learn from an expert's demonstration of performing a task~\cite{cai2019global}. It is known that this learning process can be formulated as a min-max game~\cite{ho2016generative}. In such a game the minimization is performed over all the policies and the goal is to minimize the discrepancy between the accumulated reward for expert's policy and the proposed policy. On the other hand, the maximization is done over the parameters of the reward function and aims at maximizing this discrepancy over the parameters of the reward function. This approach is also referred to as generative adversarial imitation learning (GAIL)~\cite{ho2016generative}. The problem of generative adversarial imitation learning for linear quadratic regulators~\cite{cai2019global} refers to solving this problem for the specific case where the underlying dynamic and the reward function come from a linear quadratic regulator~\cite{fazel2018global}. To be more specific, this problem can be formulated~\cite{cai2019global} as $\min_{K} \max_{\btheta\in \Theta} m(K,\btheta)$, where $K$ represents the choice of the policy and $\btheta$ represents the parameters of the dynamic and the reward functions. Under the discussed setting, $m$ is strongly concave in $\btheta$ and PL in $K$ (see~\cite{cai2019global} for more details). Note that since $m$ is strongly concave in $\btheta$ and $PL$ in $K$, any FNE of the game would also be a Nash equilibrium point. Also note that the notion of FNE does not depend on the ordering of the $\min$ and $\max$. Thus, to be consistent with our notion of PL-games, we can formulate the problem as 
\begin{align}
    \min_{\btheta \in \Theta}~\max_K -m(K, \btheta)
\end{align}
Thus, generative adversarial imitation learning of linear quadratic regulators is an example of finding a FNE for a min-max PL-game.

\end{example}

In what follows, we present a simple iterative method for computing an $\varepsilon$--FNE of PL games.

\subsection{Multi-step gradient descent ascent for PL-games}
In this section, we propose a multi-step gradient descent ascent algorithm that finds an $\varepsilon$--FNE point for PL-games. At each iteration, our method runs multiple projected gradient ascent steps to estimate the solution of the inner maximization problem. This solution is then used to estimate the gradient of the inner maximization value function, which directly provides a descent direction. In a nutshell, our proposed algorithm is a gradient descent-like algorithm on the inner maximization value function. To present the ideas of our multi-step algorithm, let us re-write \eqref{eq: game1-cons} as
\begin{align}
\min_{\btheta\in\cSt} \;\;  g(\btheta), \label{eq: game_min_uncons}
\end{align}
where 
\begin{align}
g(\btheta) \triangleq \max_{\balpha\in \cSa} f(\btheta,\balpha). \label{eq:gEval_uncons}
\end{align}

 A famous classical result in optimization is Danskin's theorem~\cite{danskin_result_1995} which provides a sufficient condition under which the gradient of the value function $\max_{\balpha \in \cSa} f(\btheta, \balpha)$ can be directly evaluated using the gradient of the objective $f(\btheta, \balpha^*)$ evaluated at the optimal solution $\balpha^*$. This result requires the optimizer $\balpha^*$ to be unique. Under our PL assumption on $f(\btheta, \cdot)$, the inner maximization problem~\eqref{eq:gEval_uncons} may have multiple optimal solutions. Hence, Danskin's theorem does not directly apply. However, as we will show in Lemma~\ref{lemma: smoothness} in the supplementary, under the PL assumption, we still can show the following result
\[\nabla_{\btheta} g(\btheta) = \nabla_{\btheta} f(\btheta, \balpha^*) \quad \mbox{with} \quad \balpha^* \in \argmax_{\balpha \in \cSa} f(\btheta, \balpha),\]
despite the non-uniqueness of the optimal solution. 

Motivated by this result, we propose a Multi-step Gradient Descent Ascent algorithm that solves the inner maximization problem to ``approximate'' the gradient of the value function $g$. This gradient direction is then used to descent on $\btheta$. More specifically, the inner loop (Step~\ref{alg1:inner-loop}) in Algorithm~\ref{alg: alg_grad} solves the maximization problem~\eqref{eq:gEval_uncons} for a given fixed value $\btheta = \btheta_t$. The computed solution of this optimization problem provides an approximation for the gradient of the function $g(\btheta)$, see Lemma~\ref{lemma: conv_alpha_main} in Appendix~\ref{app: sec-3}. This gradient is then used in Step~\ref{alg1:update-theta} to descent on $\btheta$.

\begin{algorithm}
	\caption{Multi-step Gradient  Descent Ascent} 
	\label{alg: alg_grad}
	
	\begin{algorithmic}[1]
		\State INPUT:  $K$, $T$, $\eta_1 = 1/L_{22}$, $\eta_2 = 1/L$, $\balpha_0 \in \cSa$ and $\btheta_0 \in \cSt$
		\For  {$t=0, \cdots, T-1$} \label{alg1:outer-loop}
		\State Set $\balpha_0(\btheta_t) = \balpha_t$ 
		\For {$k=0, \cdots, K-1$} \label{alg1:inner-loop}
		\State Set $\balpha_{k+1}(\btheta_t) = \balpha_{k}(\btheta_t) + \eta_1 \nabla_{\balpha}f(\btheta_t, \balpha_{k}(\btheta_t))$\label{alg1:update-alpha}
		\EndFor
		\State Set $\btheta_{t+1} = \proj_{\cSt} \Big( \btheta_t-\eta_2 \nabla_{\btheta} f(\btheta_t, \balpha_{K}(\btheta_t))\Big)$\label{alg1:update-theta}
		\EndFor
		\State Return $(\btheta_t, \balpha_{K}(\btheta_t))$~for $t=0,\cdots,T-1$.
	\end{algorithmic}
\end{algorithm}

\subsection{Convergence analysis of Multi-Step Gradient Descent Ascent Algorithm for PL games}

Throughout this section, we make the following assumption.
\begin{assumption}\label{Assumption-cons-theta}
	The constraint set $\cSt$ is convex and compact. Moreover, there exists a ball with radius $R$,  denoted by ${\cal B}_R$, such that  $\cSt \subseteq {\cal B}_R$. 
\end{assumption}

We are now ready to state the main result of this section.

\begin{theorem}\label{thm:main}
	Under Assumptions~\ref{assumption: LipSmooth-uncons} and \ref{Assumption-cons-theta}, for any given scalar $\varepsilon \in (0,1)$, if we choose $K$ and $T$ large enough such that	
	\[T \geq N_T(\varepsilon) \triangleq {\cal O}(\varepsilon^{-2}) \quad {\rm and} \quad	K \geq N_K(\varepsilon) \triangleq {\cal O}(\log\big(\varepsilon^{-1})\big),\]
	then there exists an iteration $t\in \{0,\cdots, T\}$ such that $(\btheta_t,\balpha_{t+1})$ is an $\varepsilon$--FNE of \eqref{eq: game1-cons}.
\end{theorem}

\begin{proof}
	The proof is relegated to Appendix~\ref{App: main}.
\end{proof}

\begin{corollary}
	Under Assumption~\ref{assumption: LipSmooth-uncons} and Assumption~\ref{Assumption-cons-theta}, Algorithm~\ref{alg: alg_grad} finds an $\varepsilon$-FNE of the game~\eqref{eq: game1-cons} with ${\cal O}(\varepsilon^{-2})$  gradient evaluations of the objective with respect to $\btheta$ and  ${\cal O}(\varepsilon^{-2}\log(\varepsilon^{-1}))$ gradient evaluations with respect to $\balpha$. If the two gradient oracles have the same complexity, the overall complexity of the method would be ${\cal O}(\varepsilon^{-2}\log(\varepsilon^{-1}))$. 
\end{corollary}

\begin{remark}
    The iteration complexity order $\mathcal{O}(\varepsilon^{-2}\log(\varepsilon^{-1}))$ in Theorem~\ref{thm:main} is tight (up to logarithmic factors). 
    This is due to the fact that for general non-convex smooth problems, finding an $\varepsilon$--stationary solution requires at least $\Omega(\varepsilon^{-2})$ gradient evaluations~\cite{carmon2017lower,nesterov2013introductory}. Clearly, this lower bound is also valid for finding an $\varepsilon$--FNE of PL-games. This is because we can assume that the function $f(\btheta,\balpha)$ does not depend on $\balpha$ (and thus PL in $\balpha$).
\end{remark}
\begin{remark}
	Theorem~\ref{thm:main} shows that under the PL assumption, the pair $(\btheta_t, \balpha_K(\theta_t))$ computed by Algorithm~\ref{alg: alg_grad} is an $\varepsilon$--FNE of the game~\eqref{eq: game1-cons}. Since $\balpha_K(\btheta_t)$ is an approximate solution of the inner maximization problem, we get that $\btheta_t$ is concurrently an $\varepsilon$--first order stationary solution of the optimization problem~\eqref{eq: game_min_uncons}. 
\end{remark}


\begin{remark}
	In \cite[Theorem 4.2]{sanjabi2018convergence}, a similar result was shown for the case when $f(\btheta,\balpha)$ is strongly concave in $\balpha$. Hence,  Theorem~\ref{thm:main} can be viewed as an extension of  \cite[Theorem 4.2]{sanjabi2018convergence}. Similar to \cite[Theorem 4.2]{sanjabi2018convergence}, one can easily extend the result of Theorem~\ref{thm:main} to the stochastic setting by replacing the gradient of $f$ with respect to $\theta$ in Step~\ref{alg1:update-theta} by the stochastic version of the gradient.
\end{remark}


In the next section we consider the non-convex concave min-max saddle game. It is well-known that convexity/concavity does not imply the PL condition and PL condition does not imply convexity/concavity~\cite{PL_karimi_2016}. Therefore, the problems we consider in the next section are neither restriction nor extension of our results  on PL games.

\section{Non-Convex Concave Games} \label{section:Non-Convex-Concave}
In this section, we focus on ``non-convex concave" games satisfying the following assumption:

\begin{assumption}\label{assumption:Concavity}
The objective function $f(\btheta,\balpha)$ is concave in $\balpha$ for any fixed value of $\btheta$. Moreover, the set $\cSa$ is convex and compact, and there exists a ball with radius $R$ that contains the feasible set~$\cSa$.
\end{assumption}

One major difference of this case with the PL-games is that in this case the function  $g(\btheta) = \max_{\balpha \in \cSa} \, \, f(\btheta, \balpha)$ might not be differentiable.
To see this, consider the example $ g(\alpha) =  \max_{0 \leq \alpha \leq 1} (2\alpha -1)\theta$ which is concave in $\alpha$. However, the value function~$g(\theta) = |\theta|$ is non-smooth.

Using a small regularization term, we approximate the function $g(\cdot)$ by a differentiable function
\begin{equation}\label{eq:g-lambda}
\gl(\btheta)\triangleq \max_{\balpha \in \cSa} \, \, \fl(\btheta, \balpha),
\end{equation}
where
$
\fl(\btheta, \balpha) \triangleq f(\btheta, \balpha) - \dfrac{\lambda}{2} \|\balpha - \bar{\balpha}\|^2.
$
Here $\bbalpha \in \cSa$ is some given fixed point and $\lambda>0$ is a regularization parameter that we will specify later. 
Since $f(\btheta, \balpha)$ is concave in~$\balpha$,
$\fl(\btheta, \cdot)$ is $\lambda$-strongly concave. Thus, the function $g_{\lambda}(\cdot)$ becomes smooth with Lipschitz gradient; see Lemma~\ref{lm:g-smooth} in the supplementary. Using this property, we propose an algorithm that runs at each iteration multiple steps of Nesterov accelerated projected gradient ascent to estimate the solution of~\eqref{eq:g-lambda}. This solution is then used to estimate the gradient of~$g_{\lambda} (\btheta)$ which directly provides a descent direction on $\btheta$. Our algorithm computes an $\varepsilon$--FNE point for non-convex concave games with $\widetilde{{\cal O}}(\varepsilon^{-3.5})$  gradient evaluations. 
Then for sufficiently small regularization coefficient, we show that the computed point is an $\varepsilon$-FNE. 




Notice that since $f_{\lambda}$ is Lipschitz smooth and based on the compactness assumption, we can define 
\begin{equation}\label{def:g_max}
\begin{array}{l}
g_{\btheta} \triangleq \max_{\btheta \in \cSt} \|\nabla \gl(\btheta)\|, \;\;  g_{\balpha} \triangleq \max_{\btheta \in \cSt} \|\nabla_{\balpha} \fl(\btheta, \balpha^*(\btheta) )\|, \;\; {\rm and} \;\; g_{max}=\max\{g_{\btheta}, g_{\balpha},1\},
\end{array}
\end{equation} 
where $\balpha^*(\btheta) \triangleq \argmax_{\balpha \in \cSa} \; \fl(\btheta, \balpha)$. We are now ready to describe our proposed algorithm.

\subsection{Algorithm Description}
Our proposed method is outlined in Algorithm~\ref{alg}. This algorithm has two steps: step~\ref{alg-2: update-alpha} and step~\ref{alg-2: update-theta}. In step~\ref{alg-2: update-alpha}, $K$ steps of accelerated gradient ascent method is run over the variable $\balpha$ to find an approximate maximizer of the problem $\max_{\balpha}\; f_{\lambda}(\btheta_t,\balpha)$. Then using approximate maximizer~$\balpha_{t+1}$, we update $\btheta$ variable using one step of first order methods in step~\ref{alg-2: update-theta}.

\begin{algorithm}[ht]
	\caption{Multi-Step Frank Wolfe/Projected Gradient Step Framework}\label{alg}
	\textbf{Require:} Constants $\widetilde{L}\triangleq \max\{L,L_{12}, g_{max}\}$, $N\triangleq \lfloor \sqrt{8L_{22}/\lambda}\rfloor$, $K$, $T$, $\eta$, $\lambda$, $\btheta_0 \in \cSt$,  $\balpha_0 \in \cSa$
	
	\hrulefill
	
	\begin{algorithmic}[1]
		\For{$t=0, 1, 2, \ldots, T$} \label{alg-2: outer-loop}
		\vspace{0.2cm}
		\State Set $\balpha_{t+1} = \textrm{APGA}(\balpha_t, \btheta_t, \eta,N,K)$  by running $K$ steps of  Accelerated Projected Gradient Ascent subroutine (Algorithm~\ref{alg-APGF}) with periodic restart at every $N$ iteration. \label{alg-2: update-alpha}
		\vspace{0.3cm}
		\State Compute $\btheta_{t+1}$ using first-order information (Frank-Wolfe or projected gradient descent). \label{alg-2: update-theta}
		\vspace{0.2cm}
		\EndFor
	\end{algorithmic}
\end{algorithm}

	
	

In step~\ref{alg-2: update-alpha}, we run $K$ step of accelerated gradient ascent algorithm over the variable $\alpha$ with restart every $N$ iterations. The details of this subroutine can be found  in subsection~\ref{sec.APGA} of the supplementary materials. 
In step~\ref{alg-2: update-theta} of Algorithm~\ref{alg}, we can either use projected gradient descent update rule
\[\btheta_{t+1} \triangleq \mbox{proj}_{\cSt} \Big(\btheta_t - \dfrac{1}{L_{11} + L_{12}^2/\lambda}\nabla_{\btheta} \fl \big(\btheta_t , \balpha_{t+1}\big)\Big),\]
or Frank-Wolfe update rule described in subsection~\ref{sec.FW} in the supplementary material. We show convergence of the algorithm to $\varepsilon$--FNE in Theorems~\ref{thm: FW}.

\begin{theorem}\label{thm: FW}
	Given a scalar $\varepsilon \in (0,1)$. Assume that Step~7 in Algorithm~\ref{alg} sets either runs projected gradient descent or Frank-Wolfe iteration. 
	Under Assumptions~\ref{assumption:Concavity} and \ref{assumption: LipSmooth-uncons},
	\[\eta = \dfrac{1}{L_{22}}, \quad  \lambda \triangleq \dfrac{\varepsilon}{4R}, \quad T \geq N_T(\varepsilon) \triangleq {\cal O} (\varepsilon^{-3}), \quad  {\rm and} \quad 	K \geq N_K(\varepsilon) \triangleq {\cal O}\big(\varepsilon^{-1/2} \log(\varepsilon^{-1})\big), \]
	then there exists $t \in \{0, \ldots, T\}$ such that $(\btheta_t, \balpha_{t+1})$ is an $\varepsilon$--FNE of problem~\eqref{eq: game1-cons}.
\end{theorem}

\begin{proof}
	The proof is relegated to Appendix~\ref{App: FW}.  
\end{proof}


\begin{corollary}\label{corollary: eps^{-3.5}}
	Under Assumptions~\ref{assumption: LipSmooth-uncons} and \ref{assumption:Concavity}, Algorithm~\ref{alg} finds an $\varepsilon$-first-order stationary solution of the game~\eqref{eq: game1-cons} with ${\cal O}(\varepsilon^{-3})$  gradient evaluations of the objective with respect $\btheta$ and  ${\cal O}(\varepsilon^{-0.5}\log(\varepsilon^{-1}))$ gradient evaluations with respect to $\balpha$. If the two oracles have the same complexity, the overall complexity of the method would be ${\cal O}(\varepsilon^{-3.5}\log(\varepsilon^{-1}))$.
\end{corollary}



\section{Numerical Results} \label{section:Numearical Results}
We evaluate the numerical performance of Algorithm~\ref{alg} in the following two applications:
\subsection{Fair Classifier}
We conduct two experiment on the Fashion MNIST dataset~\cite{xiao2017/online}. This dataset consists of $28\times 28$ arrays of grayscale pixel images classified into 10 categories of clothing. It includes $60,000$ training images and $10,000$ testing images.

\textbf{Experimental Setup:} The recent work in~\cite{1902.00146} observed that training a logisitic regression model to classify the images of the Fashion MNIST dataset can be biased against certain categories. 
To remove this bias, \cite{1902.00146} proposed to minimize the maximum loss incurred by the different categories. We repeat the experiment when using a more complex non-convex Convolutional Neural Network (CNN) model for classification. Similar to~\cite{1902.00146}, we limit our experiment to the three categories T-shirt/top, Coat, and Shirts, that correspond to the lowest three testing accuracies achieved by the trained classifier. To minimize the maximum loss over these three categories, we train the classifier to minimize
\begin{equation}\label{exp-non-reg}
\min_{\mathbf{W}}\,\, \max  \,\{ {\cal L}_1(\mathbf{W}),  {\cal L}_2(\mathbf{W}),  {\cal L}_3(\mathbf{W})\},
\end{equation}
where $\mathbf{W}$ represents the parameters of the CNN; and ${\cal L}_1$, ${\cal L}_2$, and ${\cal L}_3$ correspond to the loss incurred by samples in T-shirt/top, Coat, and Shirt categories. Problem~\eqref{exp-non-reg} can be re-written as
\[\min_{\mathbf{W}}\, \max_{t_1,t_2,t_3} \quad \sum_{i=1}^3 \,t_i {\cal L}_i(\mathbf{W})  \quad  
\mbox{s.t.} \quad  t_i \geq 0 \quad \forall \, i =1,2 , 3; \quad 
\sum_{i=1}^3 t_i = 1.\]
Clearly the inner maximization problem is concave; and thus our theory can be applied. To empirically evaluate the regularization scheme proposed in Section~\ref{section:Non-Convex-Concave}, we implement two versions of Algorithm~\ref{alg}. The first version solves at each iteration the regularized strongly concave sub-problem
\begin{equation}\label{finite-max-reg}
 \max_{t_1,t_2,t_3} \quad \sum_{i=1}^3 \,t_i {\cal L}_i(\mathbf{W}) - \dfrac{\lambda}{2}\sum_{i=1}^3 t_i^2 \quad 
\mbox{s.t.} \quad  t_i \geq 0 \quad \forall \, i =1,2 , 3; \quad
\sum_{i=1}^3 t_i = 1,
\end{equation}
and use the optimum $t$ to perform a gradient descent step on $\mathbf{W}$ (notice that fixing the value of $\mathbf{W}$, the optimum $t$ can be computed using KKT conditions and a simple sorting or bisection procedure). The second version of Algorithm~\ref{alg} solves at each iteration the concave inner maximization problem without the regularization term.  Then uses the computed solution to perform a descent step on $\mathbf{W}$.  Notice that in both cases, the optimization with respect to $t$ variable can be done in (almost) closed-form update. Although regularization is required to have theoretical convergence guarantees, we compare the two versions of the algorithm on empirical data to determine whether we lose by adding such regularization. We further compare these two algorithms with normal training that uses gradient descent to minimize the average loss among the three categories. We run all algorithms for $5500$ epochs and record the test accuracy of the categories. To reduce the effect of random initialization, we run our methods with $50$ different random initializations and record the average and standard deviation of the test accuracy collected. For fair comparison,  the same initialization is used for all methods in each run. The results are summarized in Tables~\ref{tab:tab-numerical-results1}. To test our framework in stochastic settings, we repeat the experiment running all algorithms for $12,000$ iterations with Adam and SGD optimizer with a bath size of $600$ images ($200$ from each category). The results of the second experiment with Adam optimizer are summarized in Table~\ref{tab:tab-numerical-results2}. The model architecture and parameters are detailed in Appendix~\ref{App: network}. The choice of Adam optimizer is mainly because it is more robust to the choice of the step-size and thus can be easily tuned. In fact, the use of SGD or Adam does not change the overall takeaways of the experiments. The results of using SGD optimizer are relegated to Appendix~\ref{App: SGD_results}.

\textbf{Results:} 
Tables~\ref{tab:tab-numerical-results1} and~\ref{tab:tab-numerical-results2} show the average and standard deviation of the number of correctly classified samples. 
The average and standard deviation are taken over 50 runs. For each run 1000 testing samples are considered for each category. The results show that when using MinMax and MinMax with regularization, the accuracies across the different categories are more balanced compared to normal training. Moreover, the tables show that Algorithm~\ref{alg} with regularization provides a slightly better worst-case performance compared to the unregularized approach. 
Note that the empirical advantages due to regularization appears more in the stochastic setting. To see this compare the differences between MinMax and MinMax with Regularization in Tables~\ref{tab:tab-numerical-results1}~and~\ref{tab:tab-numerical-results2}. Figure~\ref{fig:loss} depicts a sample trajectory of deterministic algorithm applied to the regularized and regularized formulations. This figures shows that regularization provides a smoother and slightly faster convergence compared to the unregularized approach. In addition, we apply our algorithm to the exact similar logistic regression setup as in~\cite{1902.00146}. Results of this experiment can be found in Appendix~\ref{App:logistic_results}.

\begin{figure}[]
    \centering
    \includegraphics[scale=0.3]{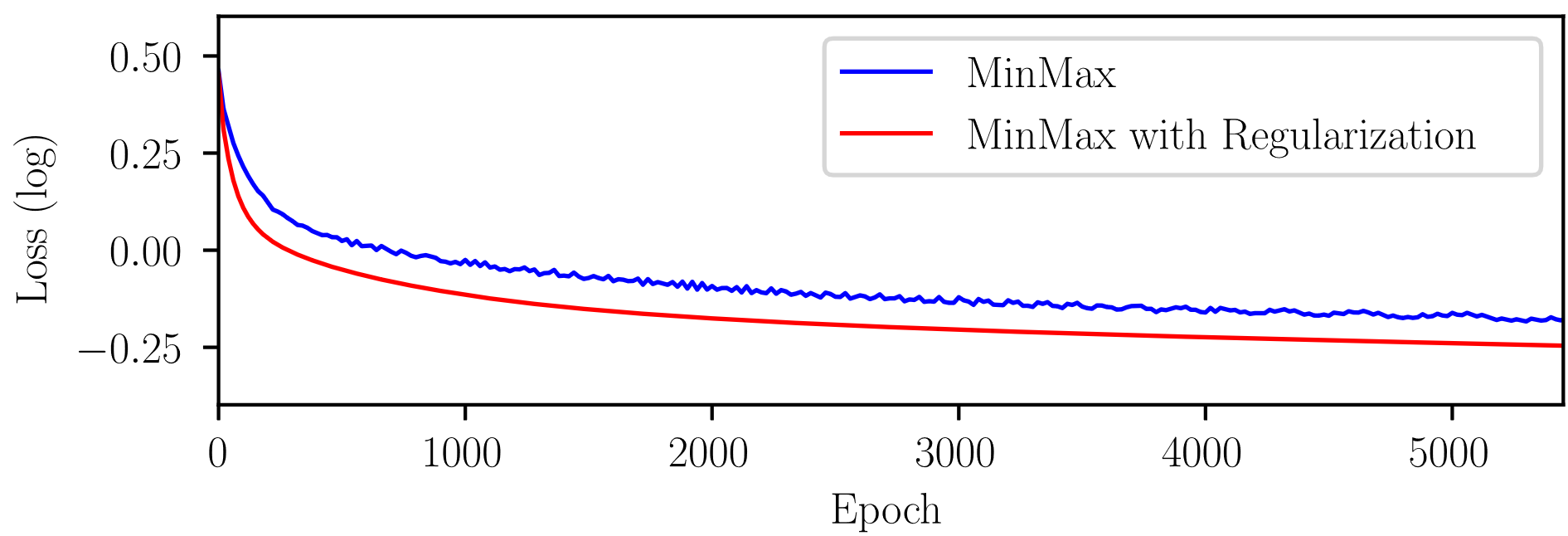}
    \caption{The effect of regularization on the convergence of the training loss, $\lambda = 0.1$.}
    \label{fig:loss}
\end{figure}
\begin{table}[h]
\centering
\resizebox{\textwidth}{!}{%
\begin{tabular}{ccccccccc}
\toprule
& \multicolumn{2}{c}{T-shirt/top} & \multicolumn{2}{c}{Coat} & \multicolumn{2}{c}{Shirt} & \multicolumn{2}{c}{Worst} \\
          & mean   & std     & mean   & std    & mean  & std   & mean & std \\ \midrule
Normal   & 850.72    & 8.58      & 843.50    & 17.24     & 658.74   & 17.81    & 658.74  & 17.81   \\ \midrule
MinMax   & 774.14    & 10.40     & 753.88    & 22.52     & 766.14   & 13.59    & 750.04  & 18.92   \\ \midrule
MinMax with Regularization & 779.84  & 10.53   & 765.56 & 22.28    & 762.34  & 11.91 & \bf{755.66}  & \bf{15.11}   \\ \bottomrule
\end{tabular}
}
\caption{The mean and standard deviation of the number of correctly classified samples when gradient descent is used in training, $\lambda=0.1$.}

\label{tab:tab-numerical-results1}
\end{table}

\begin{table}[h]
\centering
\resizebox{\textwidth}{!}{%
\begin{tabular}{ccccccccc}
\toprule
& \multicolumn{2}{c}{T-shirt/top} & \multicolumn{2}{c}{Coat} & \multicolumn{2}{c}{Shirt} & \multicolumn{2}{c}{Worst} \\
        & mean  & std    & mean  & std    & mean & std   & mean & std   \\ \midrule
Normal  & 853.86   & 10.04     & 852.22   & 18.27     & 683.32  & 17.96    & 683.32  & 17.96    \\ \midrule
MinMax  & 753.44   & 15.12     & 715.24   & 32.00     & 733.42  & 18.51    & 711.64  & 29.02    \\ \midrule
MinMax with Regularization & 764.02  & 14.12   & 739.80 & 27.60   & 748.84  & 15.79    & \bf{734.34}  & \bf{23.54}    \\ \bottomrule
\end{tabular}
}
\caption{The mean and standard deviation of the number of correctly classified samples when Adam (mini-batch) is used in training, $\lambda=0.1$.}
\label{tab:tab-numerical-results2}
\end{table}


\subsection{Robust Neural Network Training}
\textbf{Experimental Setup:} Neural networks have been widely used in various applications, especially in the field of image recognition. However, these neural networks are vulnerable to adversarial attacks, such as Fast Gradient Sign Method (FGSM)~\cite{FGSM} and Projected Gradient Descent (PGD) attack~\cite{PGD}. These adversarial attacks show that a small perturbation in the data input can significantly change the output of a neural network. To train a robust neural network  against adversarial attacks, researchers reformulate the training procedure into a robust min-max optimization formulation~\cite{madry2018}, such as
$$
\min_{\mathbf{w}}\, \; \sum_{i=1}^{N}\;\max_{\delta_i,\;\text{s.t.}\;|\delta_i|_{\infty}\leq \varepsilon}   {\ell}(f(x_i+\delta_i;\mathbf{w}), y_i).\label{eq: Madry2}
$$
Here $\mathbf{w}$ is the parameter of the neural network, the pair $(x_i,y_i)$ denotes the $i$-th data point, and $\delta_i$ is the perturbation added to data point~$i$.  
As discussed in this paper, solving such a non-convex non-concave min-max optimization problem is computationally challenging. Motivated by the theory developed in this work, we approximate the above optimization problem with a novel objective function which is concave in the parameters of the (inner) maximization player. 
To do so, we first approximate the inner maximization problem with a finite max problem
\begin{equation} \label{eq:adversaryxhatformulation}
\min_{\mathbf{w}}\, \; \sum_{i=1}^{N}\;\max \left\{   {\ell}(f(\hat{x}_{i0}(\mathbf{w});\mathbf{w}), y_i), \ldots,{\ell}(f(\hat{x}_{i9}(\mathbf{w});\mathbf{w}), y_i) \right\},
\end{equation}
where each $\hat{x}_{ij}(\mathbf{w})$ is the result of a targeted attack on sample $x_i$ aiming at changing the output of the network to  label~$j$.  These perturbed inputs, which are explained in details in Appendix~\ref{app:robust}, are the function of the weights of the network. Then we replace this finite max inner problem with a concave problem over a probability simplex. Such a concave inner problem allows us to use the multi-step gradient descent-ascent method. The structure of the network and the details of the formulation is detailed in Appendix~\ref{app:robust}.




\noindent\textbf{Results:} We compare our results with~\cite{madry2018, zhang_icml_2019}. Note~\cite{zhang_icml_2019} is the state-of-the-art algorithm and has won the first place, out of $\approx2000$ submissions, in the NeurIPS 2018 Adversarial Vision Challenge. The accuracy of our formulation against popular attacks, FGSM~\cite{FGSM} and PGD~\cite{PGD}, are summarized in Table~\ref{tab:robust_nn_results}. This table shows that our formulation leads to a comparable results  against state-of-the-art algorithms (while in some cases it also outperform those methods by as much as $\approx 15\%$ accuracy). 
\begin{table}[H]
\centering
\resizebox{\textwidth}{!}{%
\begin{tabular}{@{}rccccccc@{}}
\toprule
\multicolumn{1}{c}{} & \multirow{2}{*}{Natural} & \multicolumn{3}{c}{$\text{FGSM}\;L_{\infty}$~\cite{FGSM}} & \multicolumn{3}{c}{$\text{PGD}^{40}\; L_{\infty}$~\cite{PGD}} \\ \cmidrule(l){3-8} 
\multicolumn{1}{c}{} &  & $\varepsilon=0.2$ & $\varepsilon=0.3$ & $\varepsilon=0.4$ & $\varepsilon=0.2$ & $\varepsilon=0.3$ & $\varepsilon=0.4$ \\ \midrule
\cite{madry2018} with $\varepsilon=0.35$ & \bf{98.58}\% & 96.09\% & 94.82\% & 89.84\% & 94.64\% & 91.41\% & 78.67\% \\
\cite{zhang_icml_2019} with $\varepsilon=0.35$ & 97.37\% & 95.47\% & 94.86\% & 79.04\% & 94.41\% & 92.69\% & 85.74\% \\
\cite{zhang_icml_2019} with $\varepsilon=0.40$ & 97.21\% & 96.19\% & 96.17\% & 96.14\% & 95.01\% & 94.36\% & 94.11\% \\
Proposed with $\varepsilon=0.40$ & 98.20\% & \bf{97.04}\% & \bf{96.66}\% & \bf{96.23}\% & \bf{96.00}\% & \bf{95.17}\% & \bf{94.22}\% \\ \bottomrule
\end{tabular}%
}
\caption{Test accuracies under FGSM and PGD attacks. All adversarial images are quantified to 256 levels $\left(0-255\text{ integer}\right)$.}
\label{tab:robust_nn_results}
\end{table}
Links to code and pre-trained models of above two simulations are available at Appendix~\ref{App:links}.




\bibliography{references}

\begin{thebibliography}{10}

\bibitem{QG_anitescu2000}
M.~Anitescu.
\newblock Degenerate nonlinear programming with a quadratic growth condition.
\newblock {\em SIAM Journal on Optimization}, 10(4):1116--1135, 2000.

\bibitem{balduzzi2018mechanics}
D.~Balduzzi, S.~Racaniere, J.~Martens, J.~Foerster, K.~Tuyls, and T.~Graepel.
\newblock The mechanics of n-player differentiable games.
\newblock {\em arXiv preprint arXiv:1802.05642}, 2018.

\bibitem{beck2009fast}
A.~Beck and M.~Teboulle.
\newblock A fast iterative shrinkage-thresholding algorithm for linear inverse
  problems.
\newblock {\em SIAM journal on imaging sciences}, 2(1):183--202, 2009.

\bibitem{danskin_result_1995}
P.~Bernhard and A.~Rapaport.
\newblock On a theorem of danskin with an application to a theorem of von
  neumann-sion.
\newblock {\em Nonlinear analysis}, 24(8):1163--1182, 1995.

\bibitem{bertsekas1999nonlinear}
D.~P. Bertsekas.
\newblock {\em Nonlinear programming}.
\newblock Athena scientific Belmont, 1999.

\bibitem{cai2019global}
Q.~Cai, M.~Hong, Y.~Chen, and Z.~Wang.
\newblock On the global convergence of imitation learning: A case for linear
  quadratic regulator.
\newblock {\em arXiv preprint arXiv:1901.03674}, 2019.

\bibitem{carmon2017lower}
Y.~Carmon, J.~C. Duchi, O.~Hinder, and A.~Sidford.
\newblock Lower bounds for finding stationary points i.
\newblock {\em arXiv preprint arXiv:1710.11606}, 2017.

\bibitem{chambolle2016ergodic}
A.~Chambolle and T.~Pock.
\newblock On the ergodic convergence rates of a first-order primal--dual
  algorithm.
\newblock {\em Mathematical Programming}, 159(1-2):253--287, 2016.

\bibitem{conn2000trust}
A.~R. Conn, N.~I. Gould, and P.~L. Toint.
\newblock {\em Trust region methods}, volume~1.
\newblock Siam, 2000.

\bibitem{dai2018kernel}
B.~Dai, H.~Dai, A.~Gretton, L.~Song, D.~Schuurmans, and N.~He.
\newblock Kernel exponential family estimation via doubly dual embedding.
\newblock {\em arXiv preprint arXiv:1811.02228}, 2018.

\bibitem{dai2018sbeed}
B.~Dai, A.~Shaw, L.~Li, L.~Xiao, N.~He, Z.~Liu, J.~Chen, and L.~Song.
\newblock Sbeed: Convergent reinforcement learning with nonlinear function
  approximation.
\newblock In {\em International Conference on Machine Learning}, pages
  1133--1142, 2018.

\bibitem{dang2015convergence}
C.~D. Dang and G.~Lan.
\newblock On the convergence properties of non-euclidean extragradient methods
  for variational inequalities with generalized monotone operators.
\newblock {\em Computational Optimization and applications}, 60(2):277--310,
  2015.

\bibitem{daskalakis2017training}
C.~Daskalakis, A.~Ilyas, V.~Syrgkanis, and H.~Zeng.
\newblock Training gans with optimism.
\newblock {\em arXiv preprint arXiv:1711.00141}, 2017.

\bibitem{daskalakis2018last}
C.~Daskalakis and I.~Panageas.
\newblock Last-iterate convergence: Zero-sum games and constrained min-max
  optimization.
\newblock {\em arXiv preprint arXiv:1807.04252}, 2018.

\bibitem{daskalakis2018limit}
C.~Daskalakis and I.~Panageas.
\newblock The limit points of (optimistic) gradient descent in min-max
  optimization.
\newblock In {\em Advances in Neural Information Processing Systems}, pages
  9236--9246, 2018.

\bibitem{du2018gradientb}
S.~S. Du, J.~D. Lee, H.~Li, L.~Wang, and X.~Zhai.
\newblock Gradient descent finds global minima of deep neural networks.
\newblock {\em arXiv preprint arXiv:1811.03804}, 2018.

\bibitem{edwards2015censoring}
H.~Edwards and A.~Storkey.
\newblock Censoring representations with an adversary.
\newblock {\em arXiv preprint arXiv:1511.05897}, 2015.

\bibitem{facchinei2007finite}
F.~Facchinei and J.-S. Pang.
\newblock {\em Finite-dimensional variational inequalities and complementarity
  problems}.
\newblock Springer Science \& Business Media, 2007.

\bibitem{fazel2018global}
M.~Fazel, R.~Ge, S.~Kakade, and M.~Mesbahi.
\newblock Global convergence of policy gradient methods for the linear
  quadratic regulator.
\newblock In {\em International Conference on Machine Learning}, pages
  1466--1475, 2018.

\bibitem{ghosh2018efficient}
S.~Ghosh, M.~Squillante, and E.~Wollega.
\newblock Efficient stochastic gradient descent for distributionally robust
  learning.
\newblock {\em arXiv preprint arXiv:1805.08728}, 2018.

\bibitem{gidel2018variational}
G.~Gidel, H.~Berard, G.~Vignoud, P.~Vincent, and S.~Lacoste-Julien.
\newblock A variational inequality perspective on generative adversarial
  networks.
\newblock {\em arXiv preprint arXiv:1802.10551}, 2018.

\bibitem{gidel2018negative}
G.~Gidel, R.~A. Hemmat, M.~Pezeshki, G.~Huang, R.~Lepriol, S.~Lacoste-Julien,
  and I.~Mitliagkas.
\newblock Negative momentum for improved game dynamics.
\newblock {\em arXiv preprint arXiv:1807.04740}, 2018.

\bibitem{gidel2016frank}
G.~Gidel, T.~Jebara, and S.~Lacoste-Julien.
\newblock Frank-wolfe algorithms for saddle point problems.
\newblock {\em arXiv preprint arXiv:1610.07797}, 2016.

\bibitem{goodfellow2016deep}
I.~Goodfellow, Y.~Bengio, A.~Courville, and Y.~Bengio.
\newblock {\em Deep learning}, volume~1.
\newblock MIT press Cambridge, 2016.

\bibitem{FGSM}
I.~J. Goodfellow, J.~Shlens, and C.~Szegedy.
\newblock Explaining and harnessing adversarial examples.
\newblock {\em arXiv preprint arXiv:1412.6572}, 2014.

\bibitem{hamedani2018iteration}
E.~Y. Hamedani, A.~Jalilzadeh, N.~Aybat, and U.~Shanbhag.
\newblock Iteration complexity of randomized primal-dual methods for
  convex-concave saddle point problems.
\newblock {\em arXiv preprint arXiv:1806.04118}, 2018.

\bibitem{ho2016generative}
J.~Ho and S.~Ermon.
\newblock Generative adversarial imitation learning.
\newblock In {\em Advances in Neural Information Processing Systems}, pages
  4565--4573, 2016.

\bibitem{jin2019minmax}
C.~Jin, P.~Netrapalli, and M.~I. Jordan.
\newblock Minmax optimization: Stable limit points of gradient descent ascent
  are locally optimal.
\newblock {\em arXiv preprint arXiv:1902.00618}, 2019.

\bibitem{juditsky2016solving}
A.~Juditsky and A.~Nemirovski.
\newblock Solving variational inequalities with monotone operators on domains
  given by linear minimization oracles.
\newblock {\em Mathematical Programming}, 156(1-2):221--256, 2016.

\bibitem{PL_karimi_2016}
H.~Karimi, J.~Nutini, and M.~Schmidt.
\newblock Linear convergence of gradient and proximal-gradient methods under
  the polyak-{\l}ojasiewicz condition.
\newblock In {\em Joint European Conference on Machine Learning and Knowledge
  Discovery in Databases}, pages 795--811. Springer, 2016.

\bibitem{PGD}
A.~Kurakin, I.~Goodfellow, and S.~Bengio.
\newblock Adversarial machine learning at scale.
\newblock In {\em International Conference on Learning Representations}, 2017.

\bibitem{letcher2019differentiable}
A.~Letcher, D.~Balduzzi, S.~Racaniere, J.~Martens, J.~Foerster, K.~Tuyls, and
  T.~Graepel.
\newblock Differentiable game mechanics.
\newblock {\em Journal of Machine Learning Research}, 20(84):1--40, 2019.

\bibitem{liang2018interaction}
T.~Liang and J.~Stokes.
\newblock Interaction matters: A note on non-asymptotic local convergence of
  generative adversarial networks.
\newblock {\em arXiv preprint arXiv:1802.06132}, 2018.

\bibitem{lin2018solving}
Q.~Lin, M.~Liu, H.~Rafique, and T.~Yang.
\newblock Solving weakly-convex-weakly-concave saddle-point problems as
  weakly-monotone variational inequality.
\newblock {\em arXiv preprint arXiv:1810.10207}, 2018.

\bibitem{lu2019block}
S.~Lu, I.~Tsaknakis, and M.~Hong.
\newblock Block alternating optimization for non-convex min-max problems:
  algorithms and applications in signal processing and communications.
\newblock In {\em Proceedings of IEEE International Conference on Acoustics,
  Speech and Signal Processing (ICASSP)}, 2019.

\bibitem{lu2019hybrid}
S.~Lu, I.~Tsaknakis, M.~Hong, and Y.~Chen.
\newblock Hybrid block successive approximation for one-sided non-convex
  min-max problems: Algorithms and applications.
\newblock {\em arXiv preprint arXiv:1902.08294}, 2019.

\bibitem{madras2018learning}
D.~Madras, E.~Creager, T.~Pitassi, and R.~Zemel.
\newblock Learning adversarially fair and transferable representations.
\newblock {\em arXiv preprint arXiv:1802.06309}, 2018.

\bibitem{madry2018}
A.~Madry, A.~Makelov, L.~Schmidt, D.~Tsipras, and A.~Vladu.
\newblock Towards deep learning models resistant to adversarial attacks.
\newblock In {\em International Conference on Learning Representations}, 2018.

\bibitem{mai2018cycles}
T.~Mai, M.~Mihail, I.~Panageas, W.~Ratcliff, V.~Vazirani, and P.~Yunker.
\newblock Cycles in zero-sum differential games and biological diversity.
\newblock In {\em Proceedings of the 2018 ACM Conference on Economics and
  Computation}, pages 339--350. ACM, 2018.

\bibitem{mertikopoulos2018mirror}
P.~Mertikopoulos, H.~Zenati, B.~Lecouat, C.-S. Foo, V.~Chandrasekhar, and
  G.~Piliouras.
\newblock Mirror descent in saddle-point problems: Going the extra (gradient)
  mile.
\newblock {\em arXiv preprint arXiv:1807.02629}, 2018.

\bibitem{mescheder2018training}
L.~Mescheder, A.~Geiger, and S.~Nowozin.
\newblock Which training methods for gans do actually converge?
\newblock In {\em International Conference on Machine Learning}, pages
  3478--3487, 2018.

\bibitem{1902.00146}
M.~Mohri, G.~Sivek, and A.~T. Suresh.
\newblock Agnostic federated learning.
\newblock In {\em International Conference on Machine Learning}, pages
  4615--4625, 2019.

\bibitem{mokhtari2019unified}
A.~Mokhtari, A.~Ozdaglar, and S.~Pattathil.
\newblock A unified analysis of extra-gradient and optimistic gradient methods
  for saddle point problems: Proximal point approach.
\newblock {\em arXiv preprint arXiv:1901.08511}, 2019.

\bibitem{monteiro2010complexity}
R.~D. Monteiro and B.~F. Svaiter.
\newblock On the complexity of the hybrid proximal extragradient method for the
  iterates and the ergodic mean.
\newblock {\em SIAM Journal on Optimization}, 20(6):2755--2787, 2010.

\bibitem{nemirovski2004prox}
A.~Nemirovski.
\newblock Prox-method with rate of convergence $\mathcal{O}(1/t)$ for
  variational inequalities with lipschitz continuous monotone operators and
  smooth convex-concave saddle point problems.
\newblock {\em SIAM Journal on Optimization}, 15(1):229--251, 2004.

\bibitem{nesterov2007dual}
Y.~Nesterov.
\newblock Dual extrapolation and its applications to solving variational
  inequalities and related problems.
\newblock {\em Mathematical Programming}, 109(2-3):319--344, 2007.

\bibitem{nesterov2013introductory}
Y.~Nesterov.
\newblock {\em Introductory lectures on convex optimization: A basic course},
  volume~87.
\newblock Springer Science \& Business Media, 2013.

\bibitem{pang2016unified}
J.~S. Pang and M.~Razaviyayn.
\newblock A unified distributed algorithm for non-cooperative games., 2016.

\bibitem{pang2011nonconvex}
J.-S. Pang and G.~Scutari.
\newblock Nonconvex games with side constraints.
\newblock {\em SIAM Journal on Optimization}, 21(4):1491--1522, 2011.

\bibitem{rafique2018non}
H.~Rafique, M.~Liu, Q.~Lin, and T.~Yang.
\newblock Non-convex min-max optimization: Provable algorithms and applications
  in machine learning.
\newblock {\em arXiv preprint arXiv:1810.02060}, 2018.

\bibitem{sanjabi2018convergence}
M.~Sanjabi, J.~Ba, M.~Razaviyayn, and J.~D. Lee.
\newblock On the convergence and robustness of training gans with regularized
  optimal transport.
\newblock In {\em Advances in Neural Information Processing Systems}, pages
  7091--7101, 2018.

\bibitem{sattigeri2018fairness}
P.~Sattigeri, S.~C. Hoffman, V.~Chenthamarakshan, and K.~R. Varshney.
\newblock Fairness gan.
\newblock {\em arXiv preprint arXiv:1805.09910}, 2018.

\bibitem{sinha2018certifying}
A.~Sinha, H.~Namkoong, and J.~Duchi.
\newblock Certifying some distributional robustness with principled adversarial
  training.
\newblock {\em arXiv preprint arXiv:1710.10571}, 2017.

\bibitem{sun2018geometric}
J.~Sun, Q.~Qu, and J.~Wright.
\newblock A geometric analysis of phase retrieval.
\newblock {\em Foundations of Computational Mathematics}, 18(5):1131--1198,
  2018.

\bibitem{xiao2017/online}
H.~Xiao, K.~Rasul, and R.~Vollgraf.
\newblock Fashion-mnist: a novel image dataset for benchmarking machine
  learning algorithms.
\newblock {\em arXiv preprint arXiv:1708.07747}, 2017.

\bibitem{xu2018fairgan}
D.~Xu, S.~Yuan, L.~Zhang, and X.~Wu.
\newblock Fairgan: Fairness-aware generative adversarial networks.
\newblock In {\em 2018 IEEE International Conference on Big Data (Big Data)},
  pages 570--575. IEEE, 2018.

\bibitem{zhang_icml_2019}
H.~Zhang, Y.~Yu, J.~Jiao, E.~Xing, L.~E. Ghaoui, and M.~Jordan.
\newblock Theoretically principled trade-off between robustness and accuracy.
\newblock In {\em International Conference on Machine Learning}, pages
  7472--7482, 2019.

\end{thebibliography}
\bibliographystyle{abbrv}

\newpage
\appendix
\normalsize
\section{Proofs for results in Section~\ref{section:PL-Games}}\label{app: sec-3}
Before proceeding to the proofs of the main results,  we need some intermediate lemmas and preliminary definitions.
\begin{definition} \cite{QG_anitescu2000}
	A function $h(\bx)$ is said to satisfy the Quadratic Growth (QG) condition with constant $\gamma>0$ if 
	\[    h(\bx) - h^* \geq \dfrac{\gamma}{2}{\rm dist}(\bx)^2, \quad \forall x,\]
	where $h^*$ is the minimum value of the function, and ${\rm dist}(\bx)$ is the  distance of the point $x$ to the optimal solution set.
\end{definition}
The following lemma shows that PL implies QG \cite{PL_karimi_2016}.

\begin{lemma}[Corollary of Theorem 2 in \cite{PL_karimi_2016}]\label{lemma: QG}
	If function $f$ is PL with constant $\mu$, then $f$ satisfies the quadratic growth condition with constant $\gamma =4\mu$.
\end{lemma}
The next Lemma shows the stability of $\arg\max_{\balpha} f(\btheta,\balpha)$ with respect to $\btheta$ under PL condition.

\begin{lemma} \label{lemma: stability}
	Assume that $\{h_{\btheta}(\balpha) =-f(\btheta, \balpha) ~|~\btheta\}$ is a class of $\mu$-PL functions in $\balpha$. Define $A(\btheta) = \arg\max_{\balpha} f(\btheta,\balpha)$ and assume $A(\btheta)$ is closed. Then for any $\btheta_1$, $\btheta_2$ and $\balpha_1\in A(\btheta_1)$, there exists an $\balpha_2\in A(\btheta_2)$ such that 
	\begin{align}
	\|\balpha_1-\balpha_2\|\leq \dfrac{L_{12}}{2\mu}\|\btheta_1-\btheta_2\|
	\end{align}
\end{lemma}
\begin{proof}
	Based on the Lipchitzness of the gradients, we have that $\|\nabla_{\balpha} f(\btheta_2,\balpha_1)\|\leq L_{12}\|\btheta_1 - \btheta_2\|$. Then using the PL condition, we know that
	\begin{align}
	g(\btheta_2) + h_{\btheta_2}(\balpha_1)\leq \dfrac{L_{12}^2}{2\mu}\|\btheta_1 - \btheta_2\|^2.
	\end{align}
	Now we use the result of Lemma~\ref{lemma: QG} to show that there exists $\balpha_2 = \arg\min_{\balpha\in A(\btheta_2)}\|\balpha-\balpha_1\|^2~\in A(\btheta_2)$ such that
	\begin{align}
	2\mu\|\balpha_1-\balpha_2\|^2\leq \dfrac{L_{12}^2}{2\mu}\|\btheta_1 - \btheta_2\|^2 
	\end{align}
	re-arranging the terms, we get the desired result that \[\|\balpha_1-\balpha_2\|\leq \dfrac{L_{12}}{2\mu}\|\btheta_1-\btheta_2\|.\]
\end{proof}

Finally, the following lemma would be useful in the proof of Theorem~\ref{thm:main}. 

\begin{lemma}[See Theorem 5 in \cite{PL_karimi_2016}]\label{lemma: karimi_PL_conv}
	Assume $h(\bx)$ is $\mu$-PL and $L$-smooth. Then, by applying gradient descent with step-size $1/L$ from point $x_0$ for $K$ iterations we get an $x_K$ such that
	\begin{align}
	h(\bx) - h^*\leq \Big(1-\dfrac{\mu}{L}\Big)^K(h(x_0)-h^*),
	\end{align}
	where $h^* = \min_x h(\bx)$.
	
	
\end{lemma}
We are now ready to prove the results in Section~\ref{section:PL-Games}.

\subsection{Danskin-type Lemma for PL Functions}\label{App: smoothness}

\begin{lemma}\label{lemma: smoothness}
Under Assumption~\ref{assumption: LipSmooth-uncons} and PL-game assumption, 
	\[\nabla_{\btheta} g(\btheta) = \nabla_{\btheta} f(\btheta, \balpha^*), \quad \mbox{where} \quad \balpha^* \in \argmax_{\balpha \in \cSa} f(\btheta, \balpha).\]
	Moreover, $g$ is $L$-Lipschitz smooth with $L = L_{11} + \dfrac{L_{12}^2}{2\mu}$. 
\end{lemma}

\begin{proof}
	Let $\balpha^* \in \argmax_{\balpha \in {\cal A}} f(\btheta,\balpha)$. By Lemma~\ref{lemma: stability}, for any scalar $\tau$ and direction $d$, there exists $\balpha^*(\tau) \in \argmax_{\balpha} \; f(\btheta + \tau d,\balpha)$ such that 
	\[\|\balpha^*(\tau)-\balpha^*\|\leq \dfrac{L_{12}}{2\mu}\tau\|d\|.\]
	To find the directional derivative of $g(\cdot)$, we compute
	\[
	\begin{array}{ll}
	g(\btheta + \tau d) - g(\btheta)	&= f(\btheta+\tau d , \balpha^*(\tau)) - f(\btheta, \balpha^*) 	\\
	&= \tau\nabla_{\btheta} f(\btheta,\balpha^*)^T d + \underbrace{\nabla_{\balpha} f(\btheta,\balpha^*)^T}_{0}(\balpha^*(\tau)-\balpha^*) + \bO(\tau^2),
	\end{array}
	\]
	where the second equality holds by writing the Taylor series expansion of $f(\cdot)$.
	Thus, by definition of the directional derivative of $g(\cdot)$, we obtain
	\begin{align}
	g'(\btheta; d) = \lim_{\tau \rightarrow 0^+} \dfrac{g(\btheta + \tau d) - g(\btheta)}{\tau} = \nabla_{\btheta} f(\btheta,\balpha^*)^T d.
	\end{align} 
	Note that this relationship holds for any $d$. Thus, $\nabla g(\btheta) = \nabla_{\btheta} f(\btheta,\balpha^*)$ for any $\balpha^* \in \argmax_{\balpha \in {\cal A}} f(\btheta,\balpha) = A(\btheta)$. Interestingly, the directional derivative does not depend on the choice of $\balpha^*$. This means that $\nabla_{\btheta} f(\btheta,\balpha_1) = \nabla_{\btheta} f(\btheta,\balpha_2)$ for any $\balpha_1$ and $\balpha_2$ in $\argmax_{\balpha \in {\cal A}} f(\btheta,\balpha)$.
	
	We finally show that function $g$ is Lipschitz smooth. Let $\balpha_1^* \in A(\btheta_1)$ and $\balpha_2^* = \arg\min_{\balpha\in A(\btheta_2)}\|\balpha-\balpha^*_1\|^2~\in A(\btheta_2)$, then
	\[\arraycolsep=1pt\def\arraystretch{1.6}
	\begin{array}{ll}
	\|\nabla g(\btheta_1) - \nabla g(\btheta_2)\| &= \|\nabla_{\btheta} f(\btheta_1, \balpha_1^*) - \nabla_{\btheta} f(\btheta_2, \balpha_2^*) \|\\
	&= \|\nabla_{\btheta} f(\btheta_1, \balpha_1^*) -\nabla_{\btheta} f(\btheta_2, \balpha_1^*) +\nabla_{\btheta} f(\btheta_2, \balpha_1^*) - \nabla_{\btheta} f(\btheta_2, \balpha_2^*) \|\\
	&\leq L_{11}\|\btheta_1 - \btheta_2\| + L_{12}\|\balpha_1^* - \balpha_2^*\|\\
	&\leq \Big(L_{11} + \dfrac{L_{12}^2}{2\mu}\Big)\|\btheta_1 - \btheta_2\|,	
	\end{array}\]
	where the last inequality holds by Lemma~\ref{lemma: stability}.

\end{proof}

\subsection{Proof of Theorem~\ref{thm:main}}\label{App: main}
Using Lemma~\ref{lemma: smoothness} and Assumption~\ref{Assumption-cons-theta}, we can define 
\begin{equation}\label{def:g_theta}
\begin{array}{l}
g_{\btheta} \triangleq \max_{\btheta \in \cSt} \|\nabla g(\btheta)\| \quad {\rm and} \quad   g_{max} \triangleq \max\{g_{\btheta}, 1\}.
\end{array}
\end{equation}
The next result shows that the inner loop in Algorithm~\ref{alg: alg_grad} computes an approximate gradient of $g(\cdot)$. In other words, $\nabla_{\btheta} f(\btheta_t,\balpha_{t+1}) \approx \nabla g(\btheta_t)$.

\begin{lemma}\label{lemma: conv_alpha_main}
	Define $\kappa = \frac{L_{22}}{\mu}\geq 1$ and $\rho = 1-\frac{1}{\kappa}<1$ and assume $g(\btheta_t) - f(\btheta_t, \balpha_0(\btheta_t))<\Delta$, then for any prescribed $\varepsilon \in (0,1)$ if we choose $K$ large enough such that
	\begin{align}
	K \geq N_K(\varepsilon) \triangleq \dfrac{1}{\log{1/\rho}} \big(4\log(1/\varepsilon) +\log(2^{15} \bar{L}^6\bar{R}^6\Delta/L^2\mu) \big),
	\end{align}	
	where $\bar{L} = \max\{L_{12},L_{22},L, g_{max},1\}$ and $\bar{R} = \max\{R,1\}$, then the error $e_t \triangleq \nabla_{\btheta} f(\btheta_t, \balpha_K(\btheta_t)) - \nabla g(\btheta_t)$ has a norm 
	\begin{equation}
	\|e_t\|\leq \delta \triangleq \dfrac{L\varepsilon^2}{2^6 R(g_{max}+LR)^2}  \quad {\rm and} \quad \|\nabla_{\balpha} f(\btheta_t,\balpha_K(\btheta_t))\|\leq \varepsilon.
	\end{equation}	
\end{lemma}

\begin{proof}
	First of all, Lemma~\ref{lemma: karimi_PL_conv}  implies that
	\begin{align}
	g(\btheta_t)-f(\btheta_t, \balpha_K(\btheta_t)) \leq \rho^K\Delta.
	\end{align}
	Thus, using the QG result of Lemma~\ref{lemma: QG}, we know that there exists an $\balpha^*\in A(\btheta_t)$ such that
	\begin{align}
	\|\balpha_K(\btheta_t) - \balpha^*\| \leq \rho^{K/2}\sqrt{\dfrac{\Delta}{2\mu}}
	\end{align} 
	Thus, 
	\begin{align}
	\|e_t\|= \|\nabla_{\btheta} f(\btheta_t, \balpha_K(\btheta_t)) - \nabla g(\btheta)\|	& \leq   L_{12}\|\balpha_K(\btheta_t)-\balpha^*\|\nonumber\\
	&\leq L_{12}\rho^{K/2}\sqrt{\dfrac{\Delta}{2\mu}}\nonumber\\
	&\leq \dfrac{L\varepsilon^2}{2^6 R(g_{max}+LR)^2},
	\end{align}
	where the last inequality holds by our choice of $K$ which yields
	\[\begin{array}{l}
	\log \left(1/\rho \right)^K \geq \log \left( 1 / \varepsilon\right)^4  + \log \left(2^{15} \bar{L}^6 \bar{R}^6 \Delta / L^2\mu \right)  = \log\left(2^{15} \bar{L}^6 \bar{R}^6 \Delta / L^2\mu\varepsilon^4 \right)
\end{array}\]
which implies,
\[\begin{array}{l} {\rho}^K \leq \dfrac{2L^2\varepsilon^4 \mu}{2^{12} \Delta \bar{R}^2 \bar{L}^2(2\bar{L}\bar{R})^4} \leq \dfrac{2L^2\varepsilon^4 \mu}{2^{12} \Delta R^2 \bar{L}^2(\bar{L} + \bar{L}R)^4}
	\leq \dfrac{2L^2\varepsilon^4 \mu}{2^{12}\Delta R^2 \bar{L}^2(g_{max} + LR)^4}.
	\end{array}\]
	
	Here the second inequality holds since $\bar{R}\geq 1$, and the third inequality holds since $g_{max}\leq \bar{L}$.
	
	To prove the argument of the Lemma, note that 
	\begin{align}
	\|\nabla_{\balpha} f(\btheta_t, \balpha_K(\btheta_t)) - \underbrace{\nabla_{\balpha} f(\btheta_t,\balpha^*)}_{0}\|\leq L_{22}\|\balpha_K(\btheta_t)- \balpha^*\|\leq L_{22}\rho^{K/2}\sqrt{\dfrac{\Delta}{2\mu}}\leq \varepsilon,
	\end{align}
	where the last inequality holds by our choice of $K$ which yields
	\[{\rho}^K \leq \Big(\dfrac{\varepsilon^2 \mu}{\bar{L}^2\Delta}\Big)\underbrace{ \Big(\dfrac{\varepsilon^2L^2}{2^{15}\bar{L}^4\bar{R}^4}\Big)}_{\leq 1 } \leq \dfrac{\varepsilon^2 \mu}{\bar{L}^2\Delta}.\]
	Here the second inequality holds since $\varepsilon < 1$, $\bar{L}, \bar{R}\geq 1$, and $L\leq \bar{L}$.
	
\end{proof}

The above lemma implies that  Algorithm~\ref{alg: alg_grad} behaves similar to the simple vanilla gradient descent method applied to problem~\eqref{eq: game_min_uncons}. \\



Notice that the assumption $ g(\btheta_t)-f(\btheta_t, \balpha_0(\btheta_t))\leq \Delta, \; \forall t$ could be justified by Lemma~\ref{lemma: stability}. More specifically, by Lemma~\ref{lemma: stability}, 
\[	\|\balpha_{t+1}-\balpha_{t}\|\leq \dfrac{L_{12}}{2\mu}\|\btheta_{t+1}-\btheta_{t}\|,\]
where $\balpha_{t+1} \triangleq \argmax_{\balpha} \, f(\btheta_{t+1}, \balpha)$ and $\balpha_{t} \triangleq \argmax_{\balpha} \, f(\btheta_{t}, \balpha)$. Hence, the difference between consecutive optimal solutions computed by the inner loop of the algorithm, are upper bounded by the difference between corresponding $\btheta$'s. Since $\cSt$ is a compact set, we can find an upper bound $\Delta$ such that $ g(\btheta_t)-f(\btheta_t, \alpha_0(\btheta_t))\leq \Delta$, for all $t$. We are now ready to show Theorem~\ref{thm:main}

\begin{proof}
	We start by defining
	\[\Delta_g = g(\btheta_0)-g^*,\]
	where $g^* \triangleq \min_{\btheta} \;g(\btheta)$ is the optimal value of $g$. Note that by the compactness assumption of the set~$\cSt$, we have $\Delta_g = g(\theta_0)-g^* < \infty$.

	Based on the projection property, we know that 
	\[\big\langle \btheta_t - \dfrac{1}{L}\nabla_{\btheta}f(\btheta_t, \balpha_{t+1}) - \btheta_{t+1} , \btheta - \btheta_{t+1} \big\rangle  \leq 0 \quad \forall \; \; \btheta \in \cSt.\]
	Therefore, by setting $\btheta=\btheta_t$, we get
	\[\big\langle \nabla_{\btheta}f(\btheta_t, \balpha_{t+1}) , \btheta_{t+1} - \btheta_{t} \big\rangle  \leq -L\|\btheta_t - \btheta_{t+1}\|^2,\]
	which implies
	\begin{equation}\label{eq: PGD-PL-1}
	\arraycolsep=1pt\def\arraystretch{1.4}
	\begin{array}{ll}
	\big\langle \nabla_{\btheta}f\big(\btheta_t, \balpha^*(\btheta_t)\big) , \btheta_{t+1} - \btheta_{t} \big\rangle   &\leq -L\|\btheta_t - \btheta_{t+1}\|^2 + \big\langle \nabla_{\btheta}f\big(\btheta_t, \balpha^*(\btheta_t)\big) -\nabla_{\btheta}f\big(\btheta_t, \balpha_{t+1}\big) , \btheta_{t+1} - \btheta_{t} \big\rangle\\
	&= -L\|\btheta_t - \btheta_{t+1}\|^2 + \langle e_t,\, \btheta_t - \btheta_{t+1}\rangle
	\end{array}
	\end{equation}
	where $\balpha^*(\btheta_t) \in \argmax_{\balpha \in \cSa} \, f(\btheta_t, \balpha)$ and $e_t\triangleq \nabla_{\btheta}f\big(\btheta_t, \balpha_{t+1}\big) -\nabla_{\btheta}f\big(\btheta_t, \balpha^*(\btheta_t)\big) $. By Taylor expansion, we have
	
	\begin{equation}\label{eq: PGD-PL-2}
	\arraycolsep=1pt\def\arraystretch{1.4}
	\begin{array}{ll}
	g(\btheta_{t+1}) &\leq g(\btheta_t) + \big\langle \nabla_{\btheta} f\big(\btheta_t, \balpha^*(\btheta_t)\big), \btheta_{t+1} - \btheta_t \big\rangle+\dfrac{L}{2}\|\btheta_{t+1} - \btheta_t\|^2\\
	&\leq g(\btheta_t) -\dfrac{L}{2}\|\btheta_{t+1} - \btheta_t\|^2 + \langle e_t, \btheta_t - \btheta_{t+1}\rangle.
	\end{array}
	\end{equation}
	where the last inequality holds by \eqref{eq: PGD-PL-1}. Moreover, by the projection property, we know that
	\[\big\langle \nabla_{\btheta}f(\btheta_t, \balpha_{t+1}) , \btheta - \btheta_{t+1} \big\rangle  \geq L\big\langle \btheta_t -\btheta_{t+1}, \btheta - \btheta_{t+1} \big\rangle \quad \forall \; \btheta \in \cSt,\]
	which implies
	\begin{equation}
	\arraycolsep=1pt\def\arraystretch{1.4}
	\begin{array}{ll}
	\big\langle \nabla_{\btheta}f(\btheta_t, \balpha_{t+1}) , \btheta - \btheta_{t} \big\rangle  &\geq \big\langle \nabla_{\btheta}f(\btheta_t, \balpha_{t+1}) , \btheta_{t+1} - \btheta_t \big\rangle+ L\big\langle \btheta_t -\btheta_{t+1}, \btheta - \btheta_{t+1} \big\rangle\\\\
	& \geq -(g_{max} + 2LR + \|e_t\|)\|\btheta_{t+1} - \btheta_t\|\\\\
	& \geq -2(g_{max} + LR)\|\btheta_{t+1} - \btheta_t\|.
	\end{array}
	\end{equation}
	Here the second inequality holds by Cauchy-Schwartz, the definition of $e_t$ and our assumption that $\cSt \subseteq {\cal B}_R$. Moreover, the last inequality holds by our choice of $K$ in Lemma~\ref{lemma: conv_alpha_main} which yields
	\begin{align}
	\|e_t\|&= \|\nabla_{\btheta} f(\btheta_t, \balpha_K(\btheta_t)) - \nabla g(\btheta)\|\\
	& \leq   L_{12}\|\balpha_K(\btheta_t)-\balpha^*\|\nonumber\\
	&\leq L_{12}\rho^{K/2}\sqrt{\dfrac{\Delta}{2\mu}}\nonumber\\
	&\leq 1 \\
	&\leq g_{max}.
	\end{align}
	
	
	Hence,
	\[-\cX_t \geq -2(g_{max} + LR)\|\btheta_{t+1} - \btheta_t\|,\]
	or equivalently
	\begin{equation}
	\|\btheta_{t+1} - \btheta_t\| \geq  \dfrac{\cX_t}{2(g_{max} + LR)}.
	\end{equation}
	
	Combined with~\eqref{eq: PGD-PL-2}, we get
	\[\begin{array}{ll}g(\btheta_{t+1}) - g(\btheta_t) &\leq  -\dfrac{L}{8}\dfrac{\cX_t^2}{\big(g_{max} + LR  \big)^2} +2\|e_t\|R,
	\end{array}\]
	where the inequality holds by using Cauchy Schwartz and our assumption that $\cSt$ is in a ball of radius $R$.
	Hence,
	\[\arraycolsep=2pt\def\arraystretch{1.6}
	\begin{array}{ll}
	\dfrac{1}{T}\sum_{t=0}^{T-1}\cX_t^2 &\leq \dfrac{8\Delta_g (g_{max} + LR)^2}{LT} +\dfrac{16\delta R (g_{max} + LR)^2}{L}\\\\
	&\leq	\dfrac{\varepsilon^2}{2},
	\end{array}\]
	where the last inequality holds by using Lemma~\ref{lemma: conv_alpha_main} and choosing $K$ and $T$:
	\[T \geq N_T \triangleq \frac{32\Delta_g(g_{max} + LR)^2}{L\varepsilon^2}, \quad	K \geq N_K(\varepsilon) \triangleq \dfrac{1}{\log{1/\rho}} \big(4\log(1/\varepsilon) +\log(2^{15} \bar{L}^6\bar{R}^6\Delta/L^2\mu) \big). \]
	
	Therefore, using Lemma~\ref{lemma: conv_alpha_main}, there exists at least one index $\widehat{t}$ for which

	\begin{equation}\label{eq: PGD-PL-5}
	\cX_{\widehat{t}} \leq \varepsilon  \quad {\rm and} \quad \|\nabla_{\balpha} f(\btheta_{\widehat{t}}, \balpha_{\widehat{t}+1})\| \leq \varepsilon.
	\end{equation}
	This completes the proof of the theorem.
	
\end{proof}

\section{Algorithmic details and proofs for the results in Section~\ref{section:Non-Convex-Concave}}

\subsection{Accelerated Projected Gradient Ascent Subroutine Used in Algorithm~\ref{alg}}
\label{sec.APGA}

 \begin{algorithm}[H]
 	\caption{APGA: Accelerated Projected Gradient Ascent with Restart}\label{alg-APGF}
 	\textbf{Require:}  Constants $\balpha_t$, $\btheta_t$, $\eta$, $K$, and $N$. 
	
 	\hrulefill
	
 	\begin{algorithmic}[1]
 	    \For{$k=0,\ldots, \lfloor K/N \rfloor$}
 	        \State Set $\gamma_1 = 1$
     		\State \algorithmicif \; {$k=0$} 
     		\algorithmicthen \; $\by_1 = \balpha_t$ 
     		\algorithmicelse \; $\by_1 = \bx_N$
     		\For{$i= 1, 2, \ldots, N$} 
                \State Set $\bx_{i} = \mbox{proj}_{\cSa}\big(\by_i +\eta \nabla_{\by} \fl(\btheta_t, \by_i)\big)$
                \State Set $\gamma_{i+1} =\dfrac{1 + \sqrt{1+4\gamma_i^2}}{2}$
                \State $\by_{i+1}  = \bx_i + \Big(\dfrac{\gamma_i -1}{\gamma_{i+1}} \Big)(\bx_i - \bx_{i-1})$
     		\EndFor
 		\EndFor
 		\State Return $\bx_{N}$ 
    \end{algorithmic}
\end{algorithm}

\subsection{Frank--Wolfe update rule for Step~\ref{alg-2: update-theta} in Algorithm~\ref{alg}}
\label{sec.FW}
In Step~\ref{alg-2: update-theta} of Algorithm~\ref{alg}, instead of projected gradient descent discussed in the main body, we can also run one step of Frank--Wolfe method. More precisely, we can set
\[
\btheta_{t+1} = \btheta_t + \dfrac{\cX_t}{\widetilde{L}}\widehat{s}_t,
\]
where 
\begin{equation}\label{eq: X_t}
\begin{array}{ll} 
{\cX}_{t} \triangleq  &- \min_{s} \;\;\; \langle \nabla_{\btheta} \fl(\btheta_t, \balpha_{t+1}), s \rangle  \\ 
\\
&\quad {\st} \;\;  \btheta_t + s \in \cSt, \, \|s\|\leq 1, 
\end{array}
\end{equation}
and 
\begin{equation}\label{eq: s_t}
\begin{array}{ll} 
\widehat{s}_t \triangleq &\argmin_{s}\,\,\,\langle \nabla_{\btheta} \fl(\btheta_t, \balpha_K(\btheta_t)), s \rangle \\ 
\\
&\quad {\st} \;\;  \btheta_t + s \in \cSt, \, \|s\|\leq 1.
\end{array}
\end{equation}
is the  first order descent direction.
In the unconstrained case, the descent direction is $\widehat{s}_t = -\nabla_{\btheta} \fl(\btheta_t, \balpha_{t+1})$, which becomes the same as the gradient descent step.

\subsection{Smoothness of function $g_{\lambda}(\cdot)$}\label{App: g-smooth}

\begin{lemma}\label{lm:g-smooth}
	Under Assumption~\ref{assumption: LipSmooth-uncons} and Assumption~\ref{assumption:Concavity}, the function $\gl$ is $L$-Lipschitz smooth with $L = L_{11} + \dfrac{L_{12}^2}{\lambda}$.
\end{lemma}

\begin{proof}
	First notice that the differentiability of the function $\gl(\cdot)$ follows directly from Danskin's Theorem~\cite{danskin_result_1995}. It remains to show  that $\gl$ is a Lipschitz smooth function.
	Let 
	\[\balpha_1^* \triangleq \argmax_{\balpha \in \cSa} \;\fl (\btheta_1, \balpha)\quad {\rm  and } \quad \balpha_2^*\triangleq \argmax_{\balpha \in \cSa} \;\fl (\btheta_2, \balpha).\]
	Then by strong convexity of $-\fl(\btheta, \cdot)$, we have
	\[\begin{array}{ll}
	\fl(\btheta_2, \balpha_2^*)& \leq \fl(\btheta_2, \balpha_1^*) 
	+ \langle \nabla_{\balpha} \fl(\btheta_2, \balpha_1^*), \balpha_2^* - \balpha_1^* \rangle - \dfrac{\lambda}{2}\|\balpha_2^* - \balpha_1^*\|^2,\end{array}\]
	and
	\[\begin{array}{ll}\fl(\btheta_2, \balpha_1^*) &\leq  \fl(\btheta_2, \balpha_2^*) + \underbrace{\langle \nabla_{\balpha} \fl(\btheta_2, \balpha_2^*), \balpha_1^* - \balpha_2^* \rangle}_{\leq 0, \mbox{ by optimality of } \balpha_2^*} - \dfrac{\lambda}{2}\|\balpha_2^* - \balpha_1^*\|^2.\end{array}\]
	Adding the two inequalities, we get
	\begin{equation}\label{eq: Lipschtizness-1}
	\langle \nabla_{\balpha} \fl(\btheta_2, \balpha_1^*), \balpha_2^* - \balpha_1^* \rangle \geq \lambda\|\balpha_2^* - \balpha_1^*\|^2.
	\end{equation}
	Moreover, due to optimality of $\balpha_1^*$, we have
	\begin{equation}\label{eq: Lipschtizness-2}
	\langle  \nabla_{\balpha} \fl(\btheta_1, \balpha_1^*) , \balpha_2^* - \balpha_1^* \rangle \leq 0.
	\end{equation}
	Combining~\eqref{eq: Lipschtizness-1} and \eqref{eq: Lipschtizness-2} we obtain
	\begin{equation}\label{eq: Lipschtizness-3}
	\begin{array}{ll}
	\lambda\|\balpha_2^* - \balpha_1^*\|^2 &\leq 
	\langle  \nabla_{\balpha} \fl(\btheta_2, \balpha_1^*) - \nabla_{\balpha} \fl(\btheta_1, \balpha_1^*), \balpha_2^* - \balpha_1^* \rangle\\
	& \leq L_{12} \|\btheta_1 - \btheta_2\|\|\balpha_2^* - \balpha_1^*\|,  	
	\end{array}
	\end{equation}
	where the last inequality holds by Cauchy-Schwartz and the Lipschtizness assumption. We finally show that $\gl$ is Lipschitz smooth. 
	\[\arraycolsep=1pt\def\arraystretch{1.6}
	\begin{array}{ll}
	\|\nabla \gl(\btheta_1) - \nabla \gl(\btheta_2)\| &= \|\nabla_{\btheta} \fl(\btheta_1, \balpha_1^*) - \nabla_{\btheta} \fl(\btheta_2, \balpha_2^*) \|\\
	&= \|\nabla_{\btheta} \fl(\btheta_1, \balpha_1^*) -\nabla_{\btheta} \fl(\btheta_2, \balpha_1^*) +\nabla_{\btheta} \fl(\btheta_2, \balpha_1^*) - \nabla_{\btheta} \fl(\btheta_2, \balpha_2^*) \|\\
	&\leq L_{11}\|\btheta_1 - \btheta_2\| + L_{12}\|\balpha_1^* - \balpha_2^*\|\\
	&\leq \Big(L_{11} + \dfrac{L_{12}^2}{\lambda}\Big)\|\btheta_1 - \btheta_2\|,	
	\end{array}\]
	where the last inequality holds by~\eqref{eq: Lipschtizness-3}.	
\end{proof}

Algorithm~\ref{alg} solves the  inner maximization problem using accelerated projected gradient descent (outlined in Algorithm~\ref{alg-APGF}). The next lemma is known for accelerated projected gradient descent when applied to strongly convex functions.

\begin{lemma}\label{lm:projected-GD}
	 Assume $h(\bx)$ is $\lambda$-strongly convex and $L$-smooth. Then, applying accelerated projected gradient descent algorithm~\cite{beck2009fast} with step-size $1/L$ and restart parameter $N\triangleq \sqrt{8L/\lambda} -1$ for $K$ iterations, we get $\bx_K$ such that
	\begin{equation}\label{eq:projected-GD}
	h(\bx_K) - h(\bx^*) \leq \left(\dfrac{1}{2}\right)^{K/N}(h(\bx_0) - h(\bx^*)),
	\end{equation}
	where $\bx^* \triangleq \argmin_{\bx \in {\cal F}} h(\bx)$.
\end{lemma}

\begin{proof}
	According to \cite[Theorem~4.4]{beck2009fast}, we have
	\begin{equation}\label{eq:projected-GD-1}
	\begin{array}{ll}
	h(\bx_{iN}) - h(\bx^*) &\leq \dfrac{2L}{(N+1)^2}\|\bx_{(i-1)N} - \bx^*\|^2\\
	\\
	&\leq \dfrac{4L}{\lambda(N+1)^2}\big(h(\bx_{(i-1)N}) - h(\bx^*)\big)\\
	\\
	&\leq \dfrac{1}{2}\big(h(\bx_{(i-1)N}) - h(\bx^*)\big),
	\end{array}
	\end{equation}
	where the second inequality holds by strong convexity of $h$ and the optimality condition of $x^*$, and the last inequality holds by our choice of $N$. This yields,
	\begin{equation}\label{eq:projected-GD-2}
	h(\bx_K) - h(\bx^*) \leq (\dfrac{1}{2})^{K/N}\big(h(\bx_0) - h(\bx^*)\big),
	\end{equation}
	which completes our proof.
\end{proof}

\subsection{Proof of Theorem~\ref{thm: FW}}\label{App: FW}
We first show that the inner loop in Algorithm~\ref{alg} computes an approximate gradient of $\gl(\cdot)$. In other words, $\nabla_{\btheta} \fl(\btheta_t,\balpha_{t+1}) \approx \nabla \gl(\btheta_t)$.

\begin{lemma}\label{lm:Updating-alpha}
	Define $\kappa = \dfrac{L_{22}}{\lambda} \geq 1$ and assume $\gl(\btheta_t) - \fl(\btheta_t, \balpha_0(\btheta_t)) < \Delta$, then for any prescribed $\varepsilon \in (0,1)$ if we choose $K$ large enough such that 
	\begin{equation}\label{K-value}
	K \geq N_K(\varepsilon) \triangleq \dfrac{\sqrt{8\kappa}}{\log{2}} \big(4\log(1/\varepsilon) +\log(2^{17} \bar{L}^6\bar{R}^6\Delta/L^2\lambda) \big),
	\end{equation}
	
	where $\bar{L} \triangleq \max\{L_{12}, L_{22},L, g_{max}, 1\}$ and $\bar{R} = \max\{R,1\}$, then the error $e_t \triangleq \nabla_{\btheta} \fl(\btheta_t, \balpha_K(\btheta_t)) - \nabla \gl(\btheta)$ has a norm 
	\begin{equation}
	\|e_t\|\leq \delta \triangleq \dfrac{L\varepsilon^2}{2^6 R(g_{max}+LR)^2}
	\end{equation}
	and
	\begin{equation}
	\begin{split}
	\dfrac{\varepsilon}{2} \geq \cY_{t,K} \triangleq  &\max_{s}\,\,  \big\langle \nabla_{\balpha} \fl(\btheta_t, \balpha_K(\btheta_t)), s \big\rangle   \\
	&\quad{\st}  \quad \balpha_K(\btheta_t) + s \in \cSa, \, \|s\|\leq 1
	\end{split}.
	\end{equation}
\end{lemma}

\begin{proof}
    	Starting from Lemma~\ref{lm:projected-GD}, we have that
	\begin{equation}\label{eq:Updating-alpha-1}
	\gl(\btheta_t) - \fl (\btheta_t, \balpha_K(\btheta_t))\leq \dfrac{1}{2^{\frac{K}{\sqrt{8\kappa}}}}\Delta.
	\end{equation}

Let $\balpha^*(\btheta_t) \triangleq \argmax_{\balpha \in \cSa} \; \fl(\btheta_t, \balpha)$. Then by strong convexity of $-f(\btheta_t, \cdot)$, we get

	\begin{equation}\label{eq:Updating-alpha-2}
	\dfrac{\lambda}{2}\|\balpha_K(\btheta_t) - \balpha^*(\btheta_t)\|^2 \leq   \gl(\btheta_t) - \fl(\btheta_t, \balpha_K(\btheta_t)) \leq \dfrac{1}{2^{\frac{K}{\sqrt{8\kappa}}}}\Delta.
	\end{equation}

	Combined with the Lipschitz smoothness property of the objective, we obtain

	\begin{equation}\label{eq:Updating-alpha-3}
	\arraycolsep=1pt\def\arraystretch{2}
	\begin{array}{ll}
	\|e_t\|&=\|\nabla_{\btheta} \fl(\btheta_t, \balpha_K(\btheta_t)) - \nabla \gl(\btheta_t)\|\\
	& =\|\nabla_{\btheta} \fl(\btheta_t, \balpha_K(\btheta_t)) - \nabla_{\btheta} \fl(\btheta_t, \balpha^*(\btheta_t))\| \\
	&\leq L_{12} \|\balpha_K(\btheta_t) - \balpha^*(\btheta_t)\|\\
	&\leq  \dfrac{L_{12}}{2^{K/2\sqrt{8\kappa}}} \sqrt{\dfrac{2\Delta}{\lambda}}\\
	&\leq \dfrac{L\varepsilon^2}{2^6 R(g_{max}+LR)^2}
	\end{array}
	\end{equation}
		
	where the second inequality uses~\eqref{eq:Updating-alpha-2}, and the third inequality uses the choice of $K$ in \eqref{K-value} which yields
	yields
	\[\begin{array}{l}
	\log \left(2^{K/\sqrt{8 \kappa}}\right) \geq \log \left( 1 / \varepsilon\right)^4  + \log \left(2^{17} \bar{L}^6 \bar{R}^6 \Delta / L^2\lambda \right)  = \log\left(2^{17} \bar{L}^6 \bar{R}^6 \Delta / L^2\lambda\varepsilon^4 \right)
\end{array}\]
which implies,
\[\begin{array}{l}
\Big({\dfrac{1}{2}}\Big)^{K/2\sqrt{8 \kappa}} \leq \dfrac{L\varepsilon^2 \sqrt{\lambda}}{2^{6}\sqrt{2 \Delta} \bar{R} \bar{L}(2\bar{L}\bar{R})^2} \leq \dfrac{L\varepsilon^2 \sqrt{\lambda}}{2^{6}\sqrt{2\Delta} R \bar{L}(\bar{L} + \bar{L}R)^2}
	\leq \dfrac{L\varepsilon^2 \sqrt{\lambda}}{2^{6}\sqrt{2\Delta} R \bar{L}(g_{max} + LR)^2}.
    \end{array}\]

	Here the second inequality holds since $\bar{R}\geq 1$, and the third inequality holds since $g_{max}\leq \bar{L}$.
	To prove the second argument of the lemma, we also use the Lipschitz smoothness property of the objective to get
	\begin{equation}
	\arraycolsep=1pt\def\arraystretch{1.6}
	\begin{array}{ll}
	\big\langle \nabla_{\balpha} \fl(\btheta_t, \balpha_K(\btheta_t)), s \big\rangle	&= \big\langle \nabla_{\balpha} \fl(\btheta_t, \balpha_K(\btheta_t)) - \nabla_{\balpha} \fl(\btheta_t, \balpha^*(\btheta_t)), s \big\rangle + \big\langle \nabla_{\balpha} \fl(\btheta_t, \balpha^*(\btheta_t)), s \big\rangle\\
	&\leq \|\nabla_{\balpha} \fl(\btheta_t, \balpha_K(\btheta_t)) - \nabla_{\balpha} \fl(\btheta_t, \balpha^*(\btheta_t))\|\|s\|+\big\langle \nabla_{\balpha} \fl(\btheta_t, \balpha^*(\btheta_t)), s \big\rangle\\
	&\leq (L_{22} + \lambda)\|\balpha^*(\btheta_t) - \balpha_K(\btheta_t))\|\|s\| +\big\langle \nabla_{\balpha} \fl(\btheta_t, \balpha^*(\btheta_t)), s \big\rangle.\\
	&\leq 2L_{22}\|\balpha^*(\btheta_t) - \balpha_K(\btheta_t))\|\|s\| +\big\langle \nabla_{\balpha} \fl(\btheta_t, \balpha^*(\btheta_t)), s \big\rangle,\label{eq:Updating-alpha-4}
	\end{array}
	\end{equation}
	where the second inequality holds by our Lipschitzness assumption and the last inequality holds by our assumption that $L_{22}/\lambda \geq 1$. Moreover, 
	\begin{equation}\label{eq:Updating-alpha-5}
	\begin{array}{ll}
	\begin{array}{ll}
	\min_{s}\,\, & -\big\langle \nabla_{\balpha} \fl(\btheta_t, \balpha^*(\btheta_t)), s \big\rangle   \\
	\st \quad & \balpha_K(\btheta_t) + s \in \cSa, \, \|s\|\leq 1\end{array}	&=	\begin{array}{ll}
	\min_{\balpha}\,\, & -\big\langle \nabla_{\balpha} \fl(\btheta_t, \balpha^*(\btheta_t)), \balpha - \balpha_K(\btheta_t) \big\rangle   \\
	\st \quad & \balpha \in \cSa, \, \|\balpha - \balpha_K(\btheta_t)\|\leq 1\end{array}\\
	\\
	\\
	&=-\big\langle \nabla_{\balpha} \fl(\btheta_t, \balpha^*(\btheta_t)), \balpha^*(\btheta_t) - \balpha_K(\btheta_t) \big\rangle\\ 
	\\ & \quad - \underbrace{\begin{array}{ll}
		\max_{\balpha}\,\, & \big\langle \nabla_{\balpha} \fl(\btheta_t, \balpha^*(\btheta_t)), \balpha - \balpha^*(\btheta_t) \big\rangle  \\
		\st \quad & \balpha \in \cSa, \, \|\balpha - \balpha_K(\btheta_t)\|\leq 1\end{array}}_{0}\\
	\\
	&=-\big\langle \nabla_{\balpha} \fl(\btheta_t, \balpha^*(\btheta_t)), \balpha^*(\btheta_t) - \balpha_K(\btheta_t) \big\rangle 
	\end{array},
	\end{equation}
	where the last equality holds since $\balpha^*(\btheta_t)$ is optimal and  $\|\balpha^*(\btheta_t) - \balpha_K(\btheta_t)\| \leq 1.$ Combining~\eqref{eq:Updating-alpha-4} and \eqref{eq:Updating-alpha-5}, we get	
	\begin{equation}\label{eq:Updating-alpha-6}
	\begin{array}{l}
	\begin{array}{ll}
	\min_{s}\,\, & -\big\langle \nabla_{\balpha} \fl(\btheta_t, \balpha_K(\btheta_t)), s \big\rangle   \\
	\st \quad & \balpha_K(\btheta_t) + s \in \cSa, \, \|s\|\leq 1\end{array}\geq -\big(\|\nabla_{\balpha} \fl(\btheta_t, \balpha^*(\btheta_t))\| + 2L_{22}\big)\|\balpha_K(\btheta_t) - \balpha^*\|.
	\end{array}
	\end{equation}
	Hence, using~\eqref{def:g_max}, we get
	\begin{equation}\label{eq:Updating-alpha-7}
	\begin{array}{ll}
	\cY_{t,K}  &\leq \big( 2L_{22} + g_{max} \big)\|\balpha_K(\btheta_t) - \balpha^*\|\\
	\\
	&\leq \dfrac{3\bar{L}}{2^{K/2\sqrt{8\kappa}}}\sqrt{\dfrac{2\Delta}{\lambda}}\\
	\\
	&\leq \dfrac{\varepsilon}{2}, \\
	\end{array}
	\end{equation}
	where the second inequality uses~\eqref{eq:Updating-alpha-2}, and the last inequality holds by our choice of $K$ in \eqref{K-value} and since $\varepsilon \in (0,1)$.
\end{proof}

The above lemma implies that $ \|\nabla_{\btheta} \fl(\btheta_t, \balpha_K(\btheta_t)) - \nabla \gl(\btheta_t)\|\leq \delta \triangleq \dfrac{L\varepsilon^2}{64\bar{R}^3\bar{L}^2}  $. We now show that our assumption $ g(\btheta_t)-f(\btheta_t, \balpha_0(\btheta_t))\leq \Delta$ for all t in the above Lemma holds. Let 
\[\balpha_{t+1}^* \triangleq \argmax_{\balpha \in \cSa} \;\fl (\btheta_{t+1}, \balpha)\quad {\rm  and } \quad \balpha_t^*\triangleq \argmax_{\balpha \in \cSa} \;\fl (\btheta_t, \balpha).\]
Then by strong convexity of $-\fl(\btheta, \cdot)$, we have
\[\begin{array}{ll}
\fl(\btheta_{t+1}, \balpha_{t+1}^*)& \leq \fl(\btheta_{t+1}, \balpha_t^*) 
+ \langle \nabla_{\balpha} \fl(\btheta_{t+1}, \balpha_t^*), \balpha_{t+1}^* - \balpha_t^* \rangle - \dfrac{\lambda}{2}\|\balpha_{t+1}^* - \balpha_t^*\|^2,\end{array}\]
and
\[\begin{array}{ll}\fl(\btheta_{t+1}, \balpha_t^*) &\leq  \fl(\btheta_{t+1}, \balpha_{t+1}^*) + \underbrace{\langle \nabla_{\balpha} \fl(\btheta_{t+1}, \balpha_{t+1}^*), \balpha_t^* - \balpha_{t+1}^* \rangle}_{\leq 0, \mbox{ by optimality of } \balpha_{t+1}^*} - \dfrac{\lambda}{2}\|\balpha_{t+1}^* - \balpha_t^*\|^2.\end{array}\]
Adding the two inequalities, we get
\begin{equation}\label{eq: Lipschtizness-21}
\langle \nabla_{\balpha} \fl(\btheta_{t+1}, \balpha_t^*), \balpha_{t+1}^* - \balpha_t^* \rangle \geq \lambda\|\balpha_{t+1}^* - \balpha_t^*\|^2.
\end{equation}
Moreover, due to optimality of $\balpha_t^*$, we have
\begin{equation}\label{eq: Lipschtizness-22}
\langle  \nabla_{\balpha} \fl(\btheta_t, \balpha_t^*) , \balpha_{t+1}^* - \balpha_t^* \rangle \leq 0.
\end{equation}
Combining~\eqref{eq: Lipschtizness-21} and \eqref{eq: Lipschtizness-22} we obtain
\begin{equation}
\begin{array}{ll}
\lambda\|\balpha_{t+1}^* - \balpha_t^*\|^2 &\leq 
\langle  \nabla_{\balpha} \fl(\btheta_{t+1}, \balpha_t^*) - \nabla_{\balpha} \fl(\btheta_t, \balpha_t^*), \balpha_{t+1}^* - \balpha_t^* \rangle\\
& \leq L_{12} \|\btheta_t - \btheta_{t+1}\|\|\balpha_{t+1}^* - \balpha_t^*\|,  	
\end{array}
\end{equation}
Thus,
\[	\|\balpha_{t+1}-\balpha_{t}\|\leq \dfrac{L_{12}}{\lambda}\|\btheta_{t+1}-\btheta_{t}\|.\]
 Hence, the difference between consecutive optimal solutions computed by the inner loop of the algorithm, are upper bounded by the difference between corresponding $\btheta$'s. Since $\cSt$ is a compact set, we can find an upper bound $\Delta$ such that $ g(\btheta_t)-f(\btheta_t, \alpha_0(\btheta_t))\leq \Delta$, for all $t$.\\

 We are now ready to show the main theorem that implies convergence of our proposed algorithm to an $\varepsilon$--first-order stationary solution of problem~\eqref{eq: game1-cons}. In particular, we show that using $\nabla_{\btheta} \fl(\btheta_t, \balpha_K(\btheta_t))$ instead of $\nabla \gl(\btheta_t)$ for a small enough $\lambda$ in the Frank-Wolfe or projected descent algorithms applied to $\gl$, finds an $\varepsilon$--FNE. We are now ready to show Theorem~\ref{thm: FW}.

\begin{proof}

	\textbf{Frank-Wolfe Steps:}
	We now show the result when Step~7 of Algorithm~\ref{alg} sets
	\[\btheta_{t+1} = \btheta_t + \dfrac{{\cX}_t}{\widetilde{L}}\widehat{s}_t.\]
	Using descent lemma on $\gl$ and the definition of $\widetilde{L}$ in Algorithm~2, we have
	\begin{equation}\label{eq:FW-1}
	\begin{array}{ll}
	&\gl(\btheta_{t+1})\leq \gl(\btheta_t)  + \big\langle \nabla \gl(\btheta_t), \btheta_{t+1} - \btheta_t  \big\rangle  + \dfrac{\widetilde{L}}{2}\|\btheta_{t+1} - \btheta_t\|^2\\
	\\
	&\quad =  \gl(\btheta_t)  + \dfrac{\cX_t}{\widetilde{L}}\big\langle \nabla \gl(\btheta_t), \widehat{s}_t \big\rangle + \dfrac{\cX_t^2}{2\widetilde{L}}\|\widehat{s}_t\|^2\\
	\\
	&\quad \leq  \gl(\btheta_t)  + \dfrac{\cX_t}{\widetilde{L}}\big\langle \nabla \gl(\btheta_t), \widehat{s}_t \big\rangle + \dfrac{\cX_t^2}{2\widetilde{L}}\\
	\\
	&\quad  =  \gl(\btheta_t)  - \dfrac{\cX_t}{\widetilde{L}}\big\langle \underbrace{\nabla_{\btheta} \fl \big(\btheta_t, \balpha_K(\btheta_t)\big) - \nabla \gl(\btheta_t)}_{e_t}, \widehat{s}_t \big\rangle - \dfrac{\cX_t^2}{2\widetilde{L}}\\
	\\
	&\quad  \leq  \gl(\btheta_t)  + \dfrac{\cX_t}{\widetilde{L}}\|e_t\| - \dfrac{\cX_t^2}{2\widetilde{L}}\\
	\end{array}
	\end{equation}

	%
	%
	where $\widehat{s}_t$ and $\cX_t$  are defined in equations~\eqref{eq: X_t}~and~\eqref{eq: s_t} of the manuscript, and the second and last inequalities use the fact that $\|\widehat{s}_t\| \leq 1$.\\
	
	Summing up these inequalities for all values of $t$ leads to
	\begin{equation}\label{eq: Summing-up}
	    \dfrac{1}{T}\sum_{t=0}^{T-1} \cX_{t}^2 \leq \dfrac{2\widetilde{L}\Delta}{T} + 4\|e_t\|g_{max}\leq \dfrac{2\widetilde{L}\Delta}{T} + \dfrac{\varepsilon^2}{4} \leq \dfrac{\varepsilon^2}{2}, 
	 \end{equation}
	where the first inequality holds since 
	\[\begin{array}{ll}
	\cX_t &= \big\langle \nabla_{\btheta} \fl \big(\btheta_t, \balpha_K(\btheta_t)\big) - \nabla_{\btheta} \fl \big(\btheta_t, \balpha^*(\btheta_t)\big) + \nabla_{\btheta} \fl \big(\btheta_t, \balpha^*(\btheta_t)\big) , \widehat{s}_t \big\rangle\\
	\\
	& \leq g_{max} + \|e_t\|\\
	\\
	& \leq 2g_{max}.
	\end{array}\]
	Here the first inequality in~\eqref{eq: Summing-up} holds by \eqref{def:g_max}, Cauchy-Schwartz, and the fact that $\|\widehat{s}_t\|\leq 1$. The last inequality holds by our choice of $K$ in Lemma~\ref{lm:Updating-alpha} 
	\[K \geq N_K(\varepsilon) \triangleq \dfrac{\sqrt{8\kappa}}{\log{2}} \big(4\log(1/\varepsilon) +\log(2^{17} \bar{L}^6\bar{R}^6\Delta/L^2\lambda) \big), \]
	which yields $\|e_t\| \leq 1 \leq g_{max}$ and by choosing $T$ such that
	\[T \geq N_T(\varepsilon) \triangleq \dfrac{8\widetilde{L}\Delta}{\varepsilon^2}.\]
	Therefore, using Lemma~\ref{lm:Updating-alpha}, there exists at least one index $\widehat{t}$ for which 
	
	\begin{equation}\label{eq:FW-2}
	\cX_{\widehat{t}} \leq \varepsilon  \quad {\rm and} \quad \cY_{\widehat{t},K} \leq \dfrac{\varepsilon}{2}.
	\end{equation}
	Hence,
	
	\begin{equation}\label{eq: FW-3}
	\begin{array}{ll}
	\arraycolsep=1.5pt\def\arraystretch{1.4}
	\cY(\btheta_{\widehat{t}},\balpha_K(\btheta_{\widehat{t}})) &=
	\begin{array}{ll}
	&\max_{s}\,\, \langle \nabla_{\balpha} f(\btheta_{\widehat{t}}, \balpha_K(\btheta_{\widehat{t}})), s \rangle\\
	&\quad   \st \quad  \balpha_K(\btheta_{\widehat{t}}) + s \in \cSa, \, \|s\|\leq 1
	\end{array}
	\\
	\\
	\arraycolsep=1.5pt\def\arraystretch{1.4}
	&=\begin{array}{ll}&\max_{s}\,\, \langle \nabla_{\balpha} \fl(\btheta_{\widehat{t}}, \balpha_K(\btheta_{\widehat{t}})), s \rangle  + \lambda (\balpha_K(\btheta_{\widehat{t}}) -\bbalpha)^Ts\\
	&\quad  \st \quad  \balpha_K(\btheta_{\widehat{t}}) + s \in \cSa, \, \|s\|\leq 1
	\end{array}
	\\\\
	&\leq \cY_{{\widehat{t}},K} + \lambda\|\balpha_K(\btheta_{\widehat{t}}) -\bbalpha\|\\
	\\
	&\leq \varepsilon
	\end{array},
	\end{equation}
	where the first inequality uses Cauchy Shwartz and the fact that $\|s\|\leq 1$, and the last inequality holds due to~\eqref{eq:FW-2}, the choice of $\lambda$ in the theorem and our assumption that $\|\balpha_K(\btheta_{\widehat{t}}) -\bbalpha\| \leq 2R$.\\
	
	\vspace{0.2 in}
	
	\textbf{Projected Gradient Descent:}\\
	We start by defining
	\[\Delta_g = \gl(\btheta_0)-g^*,\]
	where $\gl^* \triangleq \min_{\btheta} \;\gl(\btheta)$ is the optimal value of $\gl$. Note that by the compactness assumption of the set~$\cSt$, we have $\Delta_g = \gl(\theta_0)-\gl^* < \infty$. \\
	
	We now show the result when Step~7 of Algorithm~\ref{alg} sets
	\[\btheta_{t+1} = \proj_{\cSt} \, \big(\btheta_t - \dfrac{1}{L}\nabla_{\btheta} f_{\lambda} (\btheta_t, \balpha_K(t)) \big),\]
	
	Based on the projection property, we know that 
	\[\big\langle \btheta_t - \dfrac{1}{L}\nabla_{\btheta}f(\btheta_t, \balpha_{t+1}) - \btheta_{t+1} , \btheta - \btheta_{t+1} \big\rangle  \leq 0 \quad \forall \; \; \btheta \in \cSt.\]
	Therefore, by setting $\btheta=\btheta_t$, we get
	\[\big\langle \nabla_{\btheta}f(\btheta_t, \balpha_{t+1}) , \btheta_{t+1} - \btheta_{t} \big\rangle  \leq -L\|\btheta_t - \btheta_{t+1}\|^2,\]
	which implies
	
	\begin{equation}\label{eq: PGD-1}
	\arraycolsep=1pt\def\arraystretch{1.4}
	\begin{array}{ll}
	\big\langle \nabla_{\btheta}f\big(\btheta_t, \balpha^*(\btheta_t)\big) , \btheta_{t+1} - \btheta_{t} \big\rangle &\leq -L\|\btheta_t - \btheta_{t+1}\|^2 + \big\langle \nabla_{\btheta}f\big(\btheta_t, \balpha^*(\btheta_t)\big) -\nabla_{\btheta}f\big(\btheta_t, \balpha_{t+1}\big) , \btheta_{t+1} - \btheta_{t} \big\rangle\\\\
	&= -L\|\btheta_t - \btheta_{t+1}\|^2 + \langle e_t,\, \btheta_t - \btheta_{t+1}\rangle
	\end{array}
	\end{equation}
	where $\balpha^*(\btheta_t) \triangleq \argmax_{\balpha \in \cSa} \, \fl(\btheta_t, \balpha)$ and $e_t\triangleq \nabla_{\btheta}f\big(\btheta_t, \balpha_{t+1}\big) -\nabla_{\btheta}f\big(\btheta_t, \balpha^*(\btheta_t)\big) $.
	
	By Taylor expansion, we have
	\begin{equation}\label{eq: PGD-2}
	\arraycolsep=1pt\def\arraystretch{1.4}
	\begin{array}{ll}
	\gl(\btheta_{t+1}) &\leq \gl(\btheta_t) + \big\langle \nabla_{\btheta} f\big(\btheta_t, \balpha^*(\btheta_t)\big), \btheta_{t+1} - \btheta_t \big\rangle  + \dfrac{L}{2}\|\btheta_{t+1} - \btheta_t\|^2\\\\
	&\leq \gl(\btheta_t) -\dfrac{L}{2}\|\btheta_{t+1} - \btheta_t\|^2 + \langle e_t, \btheta_t - \btheta_{t+1}\rangle.
	\end{array}
	\end{equation}
	
	Moreover, by the projection property, we know that
	\[\big\langle \nabla_{\btheta}f(\btheta_t, \balpha_{t+1}) , \btheta - \btheta_{t+1} \big\rangle  \geq L\big\langle \btheta_t -\btheta_{t+1}, \btheta - \btheta_{t+1} \big\rangle,\]
	which implies
	\begin{equation}
	\arraycolsep=1pt\def\arraystretch{1.4}
	\begin{array}{ll}
	\big\langle \nabla_{\btheta}f(\btheta_t, \balpha_{t+1}) , \btheta - \btheta_{t} \big\rangle  &\geq \big\langle \nabla_{\btheta}f(\btheta_t, \balpha_{t+1}) , \btheta_{t+1} - \btheta_t \big\rangle+ L\big\langle \btheta_t -\btheta_{t+1}, \btheta - \btheta_{t+1} \big\rangle\\\\
	& \geq -(g_{max} + 2LR + \|e_t\|)\|\btheta_{t+1} - \btheta_t\|\\\\
	& \geq -2(g_{max} + LR)\|\btheta_{t+1} - \btheta_t\|.
	\end{array}
	\end{equation}
	Here the second inequality holds by Cauchy-Schwartz, the definition of $e_t$ and our assumption that $\cSt \subseteq {\cal B}_R$. Moreover, the last inequality holds by our choice of $K$ in Lemma~\ref{lemma: conv_alpha_main} which yields
	\begin{align}
	\|e_t\|&= \|\nabla_{\btheta} f(\btheta_t, \balpha_K(\btheta_t)) - \nabla g(\btheta)\|\\
	& \leq   L_{12}\|\balpha_K(\btheta_t)-\balpha^*\|\nonumber\\
	&\leq L_{12}\rho^{K/2}\sqrt{\dfrac{\Delta}{2\mu}}\nonumber\\
	&\leq 1 \\
	&\leq g_{max}.
	\end{align}

	
	Hence,
	\[-\cX_t \geq -2(g_{max} + LR)\|\btheta_{t+1} - \btheta_t\|,\]
	or equivalently
	\begin{equation}
	\|\btheta_{t+1} - \btheta_t\| \geq  \dfrac{\cX_t}{2(g_{max} + LR)}.
	\end{equation}
	
	Combined with~\eqref{eq: PGD-2}, we get
	\[\begin{array}{ll}\gl(\btheta_{t+1}) - \gl(\btheta_t) &\leq  -\dfrac{L}{8}\dfrac{\cX_t^2}{\big(g_{max} + LR  \big)^2} +2\|e_t\|R,
	\end{array}\]
	where the inequality holds by using Cauchy Schwartz and our assumption that $\cSt$ is in a ball of radius $R$. 
	Hence,
	\[\arraycolsep=2pt\def\arraystretch{1.6}
	\begin{array}{ll}
	\dfrac{1}{T}\sum_{t=0}^{T-1}\cX_t^2 &\leq \dfrac{8\Delta_g (g_{max} + LR)^2}{LT} +\dfrac{16\delta R (g_{max} + LR)^2}{L}\\\\
	&\leq	\dfrac{\varepsilon^2}{2},
	\end{array}\]
	where the last inequality holds by using Lemma~\ref{lm:Updating-alpha} and choosing $K$ and $T$:
	 \[T\geq N_T(\varepsilon) \triangleq \dfrac{32\Delta_g (g_{max} + LR)^2}{L\varepsilon^2}, \quad {\rm and} \quad	K \geq N_K(\varepsilon) \triangleq \dfrac{\sqrt{8\kappa}}{\log{2}} \big(4\log(1/\varepsilon) +\log(2^{17} \bar{L}^6\bar{R}^6\Delta/L^2\lambda) \big),\]
	
	Therefore, using Lemma~\ref{lm:Updating-alpha}, there exists at least one index $\widehat{t}$ for which 
	
	\begin{equation}\label{eq: PGD-5}
	\cX_{\widehat{t}} \leq \varepsilon  \quad {\rm and} \quad \cY_{\widehat{t},K} \leq \dfrac{\varepsilon}{2}.
	\end{equation}
	Hence,

	\begin{equation}\label{eq: PGD-6}
	\begin{array}{ll}
	\arraycolsep=1.5pt\def\arraystretch{1.4}
	\cY(\btheta_{\widehat{t}},\balpha_K(\btheta_{\widehat{t}})) &= 
	\begin{array}{ll}
	&\max_{s}\,\, \langle \nabla_{\balpha} f(\btheta_{\widehat{t}}, \balpha_K(\btheta_{\widehat{t}})), s \rangle\\
	&\quad   \st \quad  \balpha_K(\btheta_{\widehat{t}}) + s \in \cSa, \, \|s\|\leq 1
	\end{array}
	\\
	\\
	&=\arraycolsep=1.5pt\def\arraystretch{1.4}
	\begin{array}{ll}&\max_{s}\,\, \langle \nabla_{\balpha} \fl(\btheta_{\widehat{t}}, \balpha_K(\btheta_{\widehat{t}})), s \rangle  + \lambda (\balpha_K(\btheta_{\widehat{t}}) -\bbalpha)^Ts\\
	&\quad  \st \quad  \balpha_K(\btheta_{\widehat{t}}) + s \in \cSa, \, \|s\|\leq 1
	\end{array}
	\\\\
	&\leq \cY_{{\widehat{t}},K} + \lambda\|\balpha_K(\btheta_{\widehat{t}}) -\bbalpha\|\\
	\\
	&\leq \varepsilon
	\end{array},
	\end{equation}
	
	where the first inequality uses Cauchy Shwartz and the fact that $\|s\|\leq 1$, and the last inequality holds due to~\eqref{eq: PGD-5}, the choice of $\lambda$ in the theorem and our assumption that $\|\balpha_K(\btheta_{\widehat{t}}) -\bbalpha\| \leq 2R$.

\end{proof}

\section{Numerical Results on Fashion MNIST with SGD}\label{App: SGD_results}

The results of using SGD optimizer are summarized in Table~\ref{tab:simulation_bs_3000_beta_5e-2} and Table~\ref{tab:simulation_bs_600_beta_5e-4}. Note SGD optimizer requires more tuning and therefore the results when $\text{batch-size}=3000$ is also included here.
\begin{table}[ht]
\centering
\resizebox{\textwidth}{!}{%
\begin{tabular}{ccccccccc}
\toprule
& \multicolumn{2}{c}{T-shirt/top} & \multicolumn{2}{c}{Coat} & \multicolumn{2}{c}{Shirt} & \multicolumn{2}{c}{Worst} \\
        & mean  & std    & mean  & std    & mean & std   & mean & std   \\ \midrule
Normal  & 850.26   & 8.59     & 806.78   & 18.92     & 558.72  & 30.99    & 558.72  & 30.99    \\ \midrule
MinMax  & 754.68   & 12.03     & 699.04   & 28.76     & 724.86  & 18.00    & 696.60  & 25.93    \\ \midrule
MinMax with Regularization & 756.16  & 13.60   & 701.02 & 30.07   & 723.14  & 18.52    & 698.16  & 26.96    \\ \bottomrule
\end{tabular}
}

\caption{The mean and standard deviation of the number of correctly classified samples when SGD (mini-batch) is used in training, $\lambda=0.05$, $ \text{batch-size}=3000$.}
\label{tab:simulation_bs_3000_beta_5e-2}
\end{table}

\begin{table}[ht]
\centering
\resizebox{\textwidth}{!}{%
\begin{tabular}{ccccccccc}
\toprule
& \multicolumn{2}{c}{T-shirt/top} & \multicolumn{2}{c}{Coat} & \multicolumn{2}{c}{Shirt} & \multicolumn{2}{c}{Worst} \\
        & mean  & std    & mean  & std    & mean & std   & mean & std   \\ \midrule
Normal  & 849.76   & 8.20     & 807.60   & 19.19     & 563.90  & 29.64    & 563.90  & 29.64    \\ \midrule
MinMax  & 755.34   & 13.72     & 702.60   & 26.11     & 723.70  & 18.92    & 700.46  & 24.02    \\ \midrule
MinMax with Regularization & 754.78  & 14.92   & 703.70 & 24.80   & 723.44  & 19.29    & 701.78  & 23.13    \\ \bottomrule
\end{tabular}
}
\caption{The mean and standard deviation of the number of correctly classified samples when SGD (mini-batch) is used in training, $\lambda=0.0005$, $ \text{batch-size}=600$.}
\label{tab:simulation_bs_600_beta_5e-4}
\end{table}

\section{Numerical Results on Fashion MNIST with Logistic Rgression Model}\label{App:logistic_results}

Table~\ref{tab:tab-numerical-results3} shows that the proposed formulation gives better accuracies under the worst category (Shirts), and the accuracies over three categories are more balanced. Note that this model is trained by gradient descent. The standard derivations not equal to $0$ is due to the early termination of the simulation.

\begin{table}[H]
\centering
\begin{tabular}{ccccccc}
\toprule
& \multicolumn{2}{c}{T-shirt/top} & \multicolumn{2}{c}{Pullover} & \multicolumn{2}{c}{Shirt} \\
          & mean   & std     & mean   & std    & mean  & std  \\ \midrule
\cite{1902.00146}   & 849.00    & 44.00      & 876.00    & 45.00 & 745.00   & 60.00      \\ \midrule
Proposed   & 778.48   & 8.78     & 773.46    & 8.76  & 740.60   & 9.26       \\ 
 \bottomrule
\end{tabular}
\caption{The mean and standard deviation of the number of correctly classified samples when gradient descent is used in training, $\lambda=0.1$.}

\label{tab:tab-numerical-results3}
\end{table}

\section{Numerical Results on Robust Neural Network Training}\label{app:robust}
Neural networks have been widely used in various applications, especially in the field of image recognition. However, these neural networks are vulnerable to adversarial attacks, such as Fast Gradient Sign Method (FGSM)~\cite{FGSM} and Projected Gradient Descent (PGD) attack~\cite{PGD}. These adversarial attacks show that a small perturbation in the data input can significantly change the output of a neural network. To train a robust neural network  against adversarial attacks, researchers reformulate the training procedure into a robust min-max optimization formulation~\cite{madry2018}, such as
$$
\min_{\mathbf{w}}\, \; \sum_{i=1}^{N}\;\max_{\delta_i,\;\text{s.t.}\;|\delta_i|_{\infty}\leq \varepsilon}   {\ell}(f(x_i+\delta_i;\mathbf{w}), y_i).\label{eq: Madry2}
$$
Here $\mathbf{w}$ is the parameter of the neural network, the pair $(x_i,y_i)$ denotes the $i$-th data point, and $\delta_i$ is the perturbation added to data point~$i$.  
As discussed in this paper, solving such a non-convex non-concave min-max optimization problem is computationally challenging. Motivated by the theory developed in this work, we approximate the above optimization problem with a novel min-max objective function which has concave inner optimization problem. 
To do so, we first approximate the inner maximization problem with a finite max problem
\begin{equation} \label{eq:adversaryxhatformulation}
\min_{\mathbf{w}}\, \; \sum_{i=1}^{N}\;\max \left\{   {\ell}(f(\hat{x}_{i0}(\mathbf{w});\mathbf{w}), y_i), \ldots,{\ell}(f(\hat{x}_{i9}(\mathbf{w});\mathbf{w}), y_i) \right\},
\end{equation}
where each $\hat{x}_{ij}(\mathbf{w})$ is the result of a targeted attack on sample $x_i$ aiming at changing the output of the network to  label~$j$. More specifically, $\hat{x}_{ij}(\mathbf{w})$ is obtained through the following procedure:

In the one but last layer of the neural network architecture for learning classification on MNIST we have 10 different neurons, each corresponding with one category of classification. For any sample~$(x_i,y_i)$ in the dataset and any $j=0,\cdots,9$, starting from $x_{ij}^0 = x_i$, we run gradient ascent to obtain the following chain of points:
$$
    x_{ij}^{k+1} = \text{Proj}_{B(x, \varepsilon)}\Big[x_{ij}^k + \alpha \nabla_x(Z_{j}(x_{ij}^k, \mathbf{w})-Z_{y_i}(x_{ij}^k, \mathbf{w}))\Big],\; k=0, \cdots, K-1,
$$
where $Z_j$ is the network logit before softmax corresponding to label $j$; $\alpha>0$ is the step-size; and $\text{Proj}_{B(x, \varepsilon)}[\cdot]$ is the projection to the infinity ball with radius $\varepsilon$ centered at $x$. Finally, we set $ \hat{x}_{ij}(\mathbf{w})= x_{ij}^K$ in \eqref{eq:adversaryxhatformulation}. 


Clearly,  we can replace the finite max  problem~\eqref{eq:adversaryxhatformulation} with a concave problem over a probability simplex, i.e., 


\begin{align}
\label{eq:finite-max2}
\min_{\mathbf{w}}\;\sum_{i=1}^{N}\;\max_{\mathbf{t} \in \mathcal{T}}\sum_{j=0}^{9}\;{t_j}\ell\left(f\left(x_{ij}^K;\mathbf{w}\right),y_i\right),\; \mathcal{T}=\{\mathbf{t}\in \mathbb{R}^{10}\,|\,\mathbf{t}\geq0,\;||\mathbf{t}||_{1}=1\},
\end{align}
which is non-convex in~$w$, but concave in~$\mathbf{t}$. Hence we can apply Algorithm~\ref{alg} to solve this opimization problem. We test~\eqref{eq:finite-max2} on MNIST dataset with a Convolutional Neural Network(CNN) with the architecture detailed in Table~\ref{tab:net-arch-robust}. The result of our experiment is presented in Table~\ref{tab:robust_nn_results_Appendix}.

\begin{table}[ht]
\centering

\begin{tabular}{@{}ll@{}}
\toprule
Layer Type           & Shape               \\ \midrule
Convolution $+$ ReLU & $5 \times 5 \times 20$     \\ 
Max Pooling          & $2 \times 2$                \\ 
Convolution $+$ ReLU & $5 \times 5 \times 50$ \\ 
Max Pooling          & $2 \times 2$             \\
Fully Connected $+$ ReLU & $800$ \\ 
Fully Connected $+$ ReLU & $500$ \\ 
Softmax & $10$ \\ \bottomrule
\end{tabular}
\caption{Model Architecture for the MNIST dataset.}
\label{tab:net-arch-robust}

\end{table}

\begin{table}[H]
\centering
\resizebox{\textwidth}{!}{%
\begin{tabular}{@{}rccccccc@{}}
\toprule
\multicolumn{1}{c}{} & \multirow{2}{*}{Natural} & \multicolumn{3}{c}{$\text{FGSM}\;L_{\infty}$~\cite{FGSM}} & \multicolumn{3}{c}{$\text{PGD}^{40}\; L_{\infty}$~\cite{PGD}} \\ \cmidrule(l){3-8} 
\multicolumn{1}{c}{} &  & $\varepsilon=0.2$ & $\varepsilon=0.3$ & $\varepsilon=0.4$ & $\varepsilon=0.2$ & $\varepsilon=0.3$ & $\varepsilon=0.4$ \\ \midrule
\cite{madry2018} with $\varepsilon=0.35$ & \bf{98.58}\% & 96.09\% & 94.82\% & 89.84\% & 94.64\% & 91.41\% & 78.67\% \\
\cite{zhang_icml_2019} with $\varepsilon=0.35$ & 97.37\% & 95.47\% & 94.86\% & 79.04\% & 94.41\% & 92.69\% & 85.74\% \\
\cite{zhang_icml_2019} with $\varepsilon=0.40$ & 97.21\% & 96.19\% & 96.17\% & 96.14\% & 95.01\% & 94.36\% & 94.11\% \\
Proposed with $\varepsilon=0.40$ & 98.20\% & \bf{97.04}\% & \bf{96.66}\% & \bf{96.23}\% & \bf{96.00}\% & \bf{95.17}\% & \bf{94.22}\% \\ \bottomrule
\end{tabular}%
}
\caption{Test accuracies under FGSM and PGD attacks. We set $K=10$ to train our model, and we take step-size 0.01 when generating PGD attacks. All adversarial images are quantified to 256 levels $\left(0-255\text{ integer}\right)$.}
\label{tab:robust_nn_results_Appendix}
\end{table}

\begin{remark}
 We would like to note that there is a mismatch between our theory and this numerical experiment. In particular, we assume smoothness of the objective function in our theory. However,	in this experiment, the ReLu activation functions  and the projection operator make the objective function non-smooth. We also did not include regularizer (strongly concave term) while solving~\eqref{eq:finite-max2} as the optimal regularizer was very small (and almost zero). 
 \end{remark}
 
\begin{remark}
 The main take away from this experiment is to demonstrate the practicality of the following idea: when solving  general challenging non-convex min-max problems, it might be possible to approximate it with one-sided non-convex min-max problems where the objective function is solvable with respect to one of the player's variable. Such a reformulation leads to computationally tractable problems and (possibly) no loss in the performance. 
\end{remark}

\section{Experimental Setup of Fair Classifier}\label{App: network}

\begin{table}[ht]
\centering
\label{tab:net-arch}
\begin{tabular}{@{}ll@{}}
\toprule
Layer Type           & Shape               \\ \midrule
Convolution $+$ tanh & $3 \times 3 \times 5$     \\ 
Max Pooling          & $2 \times 2$                \\ 
Convolution $+$ tanh & $3 \times 3 \times 10$ \\ 
Max Pooling          & $2 \times 2$             \\
Fully Connected $+$ tanh & $250$ \\ 
Fully Connected $+$ tanh & $100$ \\ 
Softmax & $3$ \\ \bottomrule
\end{tabular}
\caption{Model Architecture for the Fashion MNIST dataset. \cite{xiao2017/online}}
\end{table}

\begin{table}[ht]
\centering

\label{tab:net-para-1}
\begin{tabular}{@{}llll@{}}
\toprule
Parameter     &       &       &      \\ \midrule
Learning Rate & 0.1  & 0.05 & 0.01 \\ 
Epochs        & 4000 & 1000 & 500  \\ \bottomrule
\end{tabular}
\caption{Training Parameters for the Fashion MNIST dataset with gradient descent. \cite{xiao2017/online}}
\end{table}

\begin{table}[H]
\centering
\label{tab:net-para-2}
\begin{tabular}{@{}llll@{}}
\toprule
Parameter     &       &       &      \\ \midrule
Learning Rate & $10^{-4}$  & $10^{-5}$ & $10^{-6}$ \\ 
Iterations        & 4000 & 4000 & 4000  \\
Batch-size    & 600 && \\
\bottomrule
\end{tabular}
\caption{Training Parameters for the Fashion MNIST dataset with Adam. \cite{xiao2017/online}}
\end{table}
\begin{table}[ht]
\centering
\label{tab:net-para-3}
\begin{tabular}{@{}llll@{}}
\toprule
Parameter     &       &       &      \\ \midrule
Learning Rate & $10^{-3}$  & $10^{-4}$ & $10^{-5}$ \\ 
Iterations        & 8000 & 8000 & 8000 \\
\bottomrule
\end{tabular}
\caption{Training Parameters for the Fashion MNIST dataset with SGD. \cite{xiao2017/online}}
\end{table}

\section{Links}\label{App:links}
Robust NN Training: \url{https://github.com/optimization-for-data-driven-science/Robust-NN-Training}\\\\    
Fair Classifier: \url{https://github.com/optimization-for-data-driven-science/FairFashionMNIST}

\end{document}